\documentclass[12pt,letterpaper,reqno]{amsart}
\usepackage{amssymb,a4wide,amsthm,amsmath,amsfonts,xcolor}
\usepackage{mathrsfs}

\usepackage{verbatim}
\usepackage{hyperref}
\usepackage[utf8]{inputenc}
\usepackage[T1]{fontenc}
\newcommand{\pb}{\;\;\;}
\newcommand{\etalchar}[1]{$^{#1}$}
\usepackage{tikz}
\usetikzlibrary{calc,fadings,decorations.pathreplacing}
\usepackage{standalone}
\usepackage{pgfplots}
\usepackage{pgfplotstable}
\usepgfplotslibrary{fillbetween}
\usetikzlibrary{patterns}
\newcommand\cadlag{c{\`a}dl{\`a}g\,\,}

\newcommand{\D}{\mathrm{D}}

\DeclareMathOperator{\Leb}{Leb}
\DeclareMathOperator{\supp}{supp}

\newcommand\dela[1]{}
\newcommand\coma[1]{{\color{red} #1}} 

\newcommand{\eps}{\varepsilon}

\definecolor{ForestGreen}{rgb}{0.1,1,0.3}

\let\emptyset \undefined
\let\ge       \undefined
\let\le       \undefined
\newsymbol\le          1336  
\newsymbol\ge          133E  \let\geq\ge
\newsymbol\emptyset    203F
\newsymbol\notle       230A
\newsymbol\notge       230B
\theoremstyle{plain}
\newtheorem{theorem}{Theorem}[section]
\theoremstyle{remark}
\newtheorem{remark}[theorem]{Remark}

\theoremstyle{plain}
\newtheorem{corollary}[theorem]{Corollary}
\newtheorem{lemma}[theorem]{Lemma}
\newtheorem{Prop}[theorem]{Proposition}
\newtheorem{definition}[theorem]{Definition}

\newtheorem{assumption}[theorem]{Assumption}
\newtheorem{assumptionNotation}[theorem]{Assumption and Notation}

\newtheorem{Assumption}[theorem]{Assumption}
\numberwithin{equation}{section}

\newcommand\del[1]{}

\def\E{\mathbb E}
\def\Prob{\mathbb P}
\def\R{\mathbb R}
\def\D{\mathbb D}
\def\d{{\rm d}}

\newcommand{\C}{\mathbb{C}} 
\newcommand{\Rd}{{\mathbb{R}^d}} 
\newcommand{\Z}{\mathbb{Z}} 
\newcommand{\N}{\mathbb{N}} 
\newcommand{\F}{\mathcal{F}}
\newcommand{\Filtration}{\mathbb{F}}
\newcommand{\tildeProb}{\bar{\Prob}}
\newcommand{\Etilde}{\bar{\E}}

\newcommand{\energy}{\mathcal{E}}

\newcommand{\LinearOperators}[1]{{\mathcal{L}(#1)}}

\def\im{\mathrm i}

\newcommand\B{\mathcal{B}}
\newcommand\Bn{\mathcal{B}_n}
\newcommand{\groupBn}{e^{-\im \Bn(l) }}
\newcommand{\groupB}{e^{-\im \B(l) }}
\newcommand{\groupBnT}{e^{-\im t \Bn(l) }}
\newcommand{\groupBnS}{e^{-\im s \Bn(l) }}

\def\v{\mathbf{v}}
\newcommand{\EA}{{E_A}}
\newcommand{\cadlagBall}{\D([0,T],\mathbb{B}_{\EA}^r)}
\newcommand{\cadlagBallX}{\D([0,T],\mathbb{B}_{X}^r)}
\newcommand{\sqrtA}{A^{\frac{1}{2}}}
\newcommand{\LinftyHeins}{{L^\infty(0,T;\EA)}}
\newcommand{\norm}[1]{\Vert #1 \Vert}
\newcommand{\bigNorm}[1]{\left\Vert #1 \right\Vert}
\newcommand{\cadlagHminuseins}{{\D([0,T],\EAdual)}}

\newcommand{\EAdual}{{E_A^*}}
\newcommand{\LInfty}{{L^\infty(M)}}

\newcommand{\duality}[2]{\langle #1, #2 \rangle}
\newcommand{\dualityBig}[2]{\Big\langle #1, #2 \Big\rangle}

\newcommand{\dualitybig}[2]{\big\langle #1, #2 \big\rangle}
\newcommand{\LalphaPlusEins}{{L^{\alpha+1}(M)}}
\newcommand{\LalphaPlusEinsAlphaPlusEins}{{L^{\alpha+1}(0,T;\LalphaPlusEins)}}
\newcommand{\Lzwei}{H}

\allowdisplaybreaks
\newcommand{\skpH}[2]{\big(#1,#2\big)_{H}}
\newcommand{\skp}[2]{\big(#1,#2\big)}
\newcommand{\skpHn}[2]{\big(#1,#2\big)_{H_n}}
\newcommand{\sumM }{\sum_{m=1}^{N}}
\newcommand{\skpLzwei}[2]{\big(#1,#2\big)_{L^2}}
\newcommand{\Real}{\operatorname{Re}}
\newcommand{\df }{\mathrm{d}}
\newcommand{\weaklyCadlag}[1]{\D_w\left([0,T],#1\right)}

\newcommand{\LalphaPlusEinsDual}{{L^{\frac{\alpha+1}{\alpha}}(M)}}
\newcommand{\Fhat}{\hat{F}}

\newcommand{\countingMeasures}{M^\nu_{\bar{\mathbb{N}}}([0,T]\times \R^M)}

\newcommand{\Id}{\operatorname{Id}}
\begin{document}
	\title[Stochastic NLS Driven by Pure Jump Noise]
	{Weak Martingale solutions for the stochastic nonlinear Schr\"odinger equation driven by pure jump noise}
	\author[Z.~Brze\'zniak]{Zdzis{\l}aw Brze\'zniak}
	\address{Department of Mathematics\\
		The University of York\\
		Heslington, York YO10 5DD, UK} \email{zdzislaw.brzezniak@york.ac.uk}
	\author[F.~Hornung]{Fabian Hornung}
	\address{Institute for Analysis\\ Karlsruhe Institute for Technology (KIT)\\ 76128 Karlsruhe, Germany}
	\email{fabian.hornung@kit.edu}
	\author[U.~Manna]{Utpal Manna}
	\address{School of Mathematics\\ Indian Institute of Science Education and Research Thiruvananthapuram\\ Trivandrum 695016, INDIA}
	\email{manna.utpal@iisertvm.ac.in}
\thanks{}
	\begin{abstract}
		We construct a martingale solution of the stochastic nonlinear Schr\"odinger equation with a multiplicative noise of jump type in the Marcus canonical form. The problem is formulated in a general framework that covers the subcritical focusing and defocusing stochastic NLS in $H^1$ on compact manifolds and on bounded domains with various boundary conditions.  The proof is based on a variant of the Faedo-Galerkin method. In the formulation of the approximated equations, finite dimensional operators derived from the Littlewood-Paley decomposition complement the classical orthogonal projections to guarantee uniform estimates. Further ingredients of the construction are tightness criteria in certain spaces of \cadlag functions and Jakubowski's generalization of the Skorohod-Theorem to nonmetric spaces.
\end{abstract}
\date{\today}
\maketitle
\date{\today}
\maketitle
\tableofcontents

\newtheorem{sta}{Statement}
\newtheorem{pro}{Proposition}
\newtheorem{thm}[pro]{Theorem}
\newtheorem{lem}[pro]{Lemma}
\newtheorem{cor}[pro]{Corollary}

\newtheorem{ass}[pro]{Assumption}

\newtheorem{defn}[pro]{Definition}

\newtheorem{rem}{Remark}\tableofcontents

\section{Introduction}
In this paper, we study the stochastic nonlinear Schr\"{o}dinger equation with pure jump noise in the Marcus form
\begin{equation}
\label{ProblemMarcus}
\left\{
\begin{aligned}
\d u(t)&= \left(-\im A u(t)-\im  F(u(t))\right) \df t-\im \sumM B_mu(t) \diamond \df L_m(t) \qquad t> 0,\\
u(0)&=u_0.
\end{aligned}\right.
\end{equation}
Here, $A$ is a selfadjoint nonnegative operator with a compact resolvent in an $L^2$-space $H$ and the initial value $u_0$ is chosen from the energy space $E_A:=\mathcal{D}(\sqrtA).$ Typical examples for this setting are
\begin{itemize}
	\item the negative Laplace-Beltrami operator $A=-\Delta_g$ on a compact riemannian manifold $(M,g)$ without boundary, $\EA=H^1(M),$
	\item the negative Laplacian $A=-\Delta$ on a bounded domain $M\subset \Rd$ with Neumann boundary condition, i.e. $\EA=H^1(M),$ or Dirichlet boundary conditions, i.e. $\EA=H^1_0(M)$
	\item and fractional powers of the first two examples.
\end{itemize}
Moreover, $F: \EA \to \EAdual$ is a nonlinear map generalizing the two most important examples, namely
\begin{itemize}
	\item
	the defocusing power nonlinearity $F_{\alpha}^+(u):=\vert u\vert^{\alpha-1}u$ with subcritical exponents in the sense that the embedding $\EA \hookrightarrow L^{\alpha+1}$ is compact
	\item
	and the focusing  nonlinearity $F_{\alpha}^-(u):=-\vert u\vert^{\alpha-1}u$ with an additional restriction to the power $\alpha.$
\end{itemize}
The stochastic noise term is given by selfadjoint
linear bounded operators $B_m$ for $m=1,\dots, N$ and an $\R^N-$ valued L\'evy process $L(t) := (L_1(t), \cdots, L_N(t))$ with pure jump defined as
\begin{equation}\label{smallJumpLevy}
	L(t) = \int_{0}^{t}\int_{B}\!l \,\tilde{\eta}(\df s,\df l) 
\end{equation}
 where $B := \left\{\vert l\vert\le 1\right\} \subset \R^N.$ Here, $\eta$ represents a time homogeneous Poisson random measure with $\sigma$-finite intensity measure $\nu$ such that 
	\begin{align*}
 \int_B \vert l\vert^2 \nu(\df l)<\infty.
 \end{align*}
 Moreover, $\tilde{\eta}:=\eta-\Leb\otimes\nu$ denotes the corresponding time homogeneous compensated Poisson random measure (see Appendix \ref{PRM} for details).   
Note that by the choice of $L$ in \eqref{smallJumpLevy}, we restrict ourselves to the case of small jumps. A generalization of the results of the present article to noise with jumps of arbitrary size will be investigated.
Using the abbreviation \begin{align*}
\B(l)=\sum_{m=1}^N l_m B_m, \qquad l\in \R^N,
\end{align*}
the equation \eqref{ProblemMarcus} including the Marcus product $\diamond$ is understood in the sense of the associated integral equation
\begin{align}\label{marcusIntroIntegralForm}
u(t)  & =  u_0 -\im \int_0^t\left( A u(s)+F(u(s))\right) \,\df s + \int_{0}^{t}\! \int_{B} \!\left[\groupB u(s-) - u(s-)\right]\, \tilde{\eta}(\df s,\df l)\nonumber\\
&\quad + \int_{0}^{t}\! \int_{B} \! \left\{\groupB u(s) - u(s) + \im \B(l) u(s)\right\}\, \nu(\df l)\df s. 
\end{align} 
Before we describe our approach and state our result in detail, we would like to give a general overview of the literature on the stochastic NLS. In the two previous decades, existence and uniqueness results for the stochastic NLS with Gaussian noise have been treated in many articles, most notably  \cite{BouardLzwei}, \cite{BouardHeins},\cite{BarbuL2},\cite{BarbuH1},\cite{FHornung} in the $\Rd$-setting, \cite{BrzezniakStrichartz} for general $2D$ compact manifolds and \cite{CheungMosincat} for the $d$-dimensional torus $\mathbb{T}^d.$ In these articles, the authors applied Strichartz estimates in a fixed point argument based on the mild formulation. Typically, this argument was either combined with a transformation to a random NLS without stochastic integral or with a truncation of the nonlinearities and suitable estimates of stochastic convolutions.

In their joint papers \cite{ExistencePaper} and \cite{UniquenessPaperWeis} together with Lutz Weis, the first and second named author developed a different approach to the stochastic NLS with Gaussian noise. By complementing the classical Faedo-Galerkin approximation with methods from spectral theory and particularly, a general version of the Littlewood-Paley decomposition, they were able to prove the existence of a martingale solution. In contrast to the argument based on Strichartz estimates, the construction only employs the Hamiltonian structure of the NLS and certain compact Sobolev embeddings. Therefore, the result could be formulated in a rather general setting including the stochastic NLS and the stochastic fractional NLS on compact manifolds and bounded domains. Subsequently, the authors concentrated on the special case of $2D$ manifolds with bounded geometry and $3D$ compact manifolds and proved pathwise uniqueness using appropriate Strichartz estimates from \cite{Burq} and \cite{Bernicot}. For a slight generalization of the existence result from \cite{ExistencePaper} allowing a certain class of non-conservative nonlinear noise, we refer to the PhD thesis \cite{FHornungPhD} of the second author.

In contrast to their Gaussian counterpart, stochastic nonlinear Schr\"odinger equations with jump noise as in \eqref{ProblemMarcus} are less well studied in the literature.
Models of this type have been proposed in \cite{Montero2010} and \cite{Montero2011} to incorporate  amplification of a signal in a fiber at random isolated locations caused by material inhomogeneities. In \cite{BouardHausenblasExistence}, de Bouard and Hausenblas considered a similar problem as \eqref{ProblemMarcus} on the full space $\Rd$ and obtained the existence of a martingale solution. The authors  continued their work and in the recent preprint \cite{BouardHausenblasUniqueness} with Ondrejat, and proved pathwise uniqueness in the $\Rd$-setting. 
The analysis of the noise in our present work is different compared to \cite{BouardHausenblasExistence, BouardHausenblasUniqueness} and is motivated by the requirement that the noise must preserve the invariance property under coordinate transformation. This issue is important for the norm-preserving condition, see \eqref{eqn-normPreservation} below. Thus, one needs to find an analogue of the Stratonovich integral in the case of stochastic integral with respect to compensated Poisson random measure.  The work of Marcus \cite{Marcus}, developed later by  Applebaum and  Kunita, see e.g. Section 6.10 of Applebaum \cite{Applebaum_2009} and Kunita \cite{Kunita}; see also Chechkin and Pavlyukevich \cite{ChPa}; provides a framework to resolve this technical issue. Surprisingly, the literature on stochastic partial differential equations driven by L\'evy noise in the ``Marcus" canonical form is very limited and such work has recently been initiated by the first and third named authors in \cite{Brz+Manna_2017_SLLGELevy} for the Landau-Lifshitz-Gilbert equation. The current paper is motivated by similar question and we believe that the theory developed in this work may help in understanding analysis of many other constrained PDEs (e.g. harmonic map flow, nematic liquid crystal model etc.) driven by jump noise or more general L\'evy noise. Also, there are some very recent works, see e.g. Chevyrev and Friz \cite{CF}, where rough differential equations are studied in the spirit of Marcus canonical stochastic differential equations by dropping the assumption of continuity prevalent in the rough path literature. Therefore, we hope that Gubinelli's \cite{Gubinelli} approach of Lyons' theory of integration over rough paths may be integrated with \cite{CF} and our approach to gain newer insight into the analysis of constrained SPDEs.

The goal of the present study is to construct a martingale solution of the stochastic NLS with pure jump noise in the Marcus canonical form. For that purpose, we transfer the argument developed in \cite{ExistencePaper} for the NLS with Gaussian noise to the present setting. Let us present our reasoning in detail. First, we introduce a strictly positive operator $S$ which commutes with $A$ and also has a compact resolvent. The operator $S$ is used to present a unified proof for each example and will chosen individually in the different concrete settings from Section 3. Typical choices are $S=A$ or $S=\Id+A$.
By means of the functional calculus of $S$ which is based on its series representation, we define operators $P_n=p_n(S)$ and $S_n=s_n(S)$ for $n\in\N_0.$  The functions $p_n$ and $s_n$, $n\in\N_0$ are illustrated in Figure \ref{figureSn}. For the precise definition, we refer to Section 5 and particularly the proof of Proposition \ref{PaleyLittlewoodLemma}. To summarize the most important properties of these operators, we remark that both $P_n$ and $S_n$ have a finite dimensional range, $P_n$ is an orthogonal projection and the operators $S_n$ satisfy the uniform estimate
$\sup_{n\in\N_0}\norm{S_n}_{\LinearOperators{L^{\alpha+1}}}<\infty$ since we assume that $S$ satisfies (generalized) Gaussian bounds.
Let us remark that a similar construction has  been employed in \cite{hornung2017strong} to construct a solution of a stochastic nonlinear Maxwell equation with Gaussian noise. This indicates that using operators like $S_n,$ $n\in\N_0$,  significantly increases the field of application of the classical Faedo-Galerkin method for both continuous and jump noise.

\begin{figure}[h]\label{figureSn}
	\begin{minipage}{0.49\textwidth}

\begin{tikzpicture}


\draw[->] (-0.2,0) -- (5,0) node[anchor=north] {$\lambda$};
\draw[->] (0,-0.2) -- (0,2) node[anchor=east] {$p_n(\lambda)$};


\draw	(-0.2,0) node[anchor=east] {$0$}
(-0.2,1.5) node[anchor=east] {$1$}
(0,-0.2) node[anchor=north] {$0$}
(2,-0.2) node[anchor=north] {$2^n$}
(4,-0.2) node[anchor=north] {$2^{n+1}$};

\draw (2,-0.1)--(2,0.1);
\draw (4,-0.1)--(4,0.1);


\draw [line width=1pt](0.014,1.5)--(4,1.5);
\draw [ dashed,line width=0.5pt](4,1.5)--(4,0.002);
\draw [line width=1pt](3.991,0)--(4.95,0);
\end{tikzpicture}
\end{minipage}
\begin{minipage}{0.49\textwidth}

\begin{tikzpicture}

\draw[->] (-0.2,0) -- (5,0) node[anchor=north] {$\lambda$};
\draw[->] (0,-0.2) -- (0,2) node[anchor=east] {$s_n(\lambda)$};


\draw	(-0.2,0) node[anchor=east] {$0$}
(-0.2,1.5) node[anchor=east] {$1$}
(0,-0.2) node[anchor=north] {$0$}
(2,-0.2) node[anchor=north] {$2^n$}
(4,-0.2) node[anchor=north] {$2^{n+1}$};
\draw (2,-0.1)--(2,0.1);
\draw (4,-0.1)--(4,0.1);


\draw [line width=1pt](0.014,1.5)--(2,1.5);
\draw [line width=1pt](3.991,0)--(4.99,0);
\draw[domain=2:4, variable=\x,line width=1pt] plot (\x,{1.5*(-6*(0.5*\x-1)^5+15*(0.5*\x-1)^4-10*(0.5*\x-1)^3+1)});

\end{tikzpicture}
\end{minipage}
\caption{Plot of the functions $p_n$ and $s_n$}\label{FigureSnPn}
\end{figure}
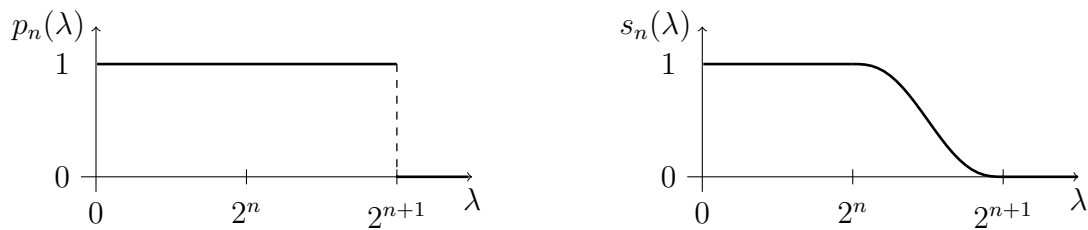 


Let us denote $\Bn(l)=\sum_{m=1}^N l_m S_n B_m S_n$ for  $n\in \N$ and $l\in \R^N$ and 
	\begin{align*}
\widetilde{u_{0,n}}:=
\begin{cases}
S_n u_0 \frac{\norm{u_0 }_H}{\norm{S_n u_0 }_H}, &  S_n u_0\neq 0,\\
0,& S_n u_0=0.		
\end{cases}
\end{align*}
for $n\in\N.$
Then, the finite dimensional approximation 
\begin{align}\label{marcusGalerkinIntro}
u_n(t)  & = P_n u_0 -\im \int_0^t\left( A u_n(s)+P_n F(u_n(s))\right) \,\df s
+ \int_{0}^{t}\! \int_{B} \!\left[\groupBn u_n(s-) - u_n(s-)\right]\, \tilde{\eta}(\df s,\df l)\nonumber\\
&\quad + \int_{0}^{t}\! \int_{B} \! \left\{ \groupBn u_n(s) - u_n(s) + \im \Bn(l) u_n(s)\right\}\, \nu(\df l)\df s
\end{align}
of problem \eqref{ProblemMarcus} has a unique solution. Due to the properties of $P_n$ and $S_n$
and the Hamiltonian structure of the nonlinear Schrödinger equation combined with the Marcus structure of the noise, we are able to prove the mass identity
\begin{align*}
	\norm{u_n(t)}_{L^2}=\norm{P_n u_0}_{L^2}
\end{align*} 
almost surely for all $t\in[0,T]$
and the uniform estimate 
\begin{align}\label{energyEstimateIntroduction}
	\sup_{n\in\N}\E \Big[\sup_{t\in[0,T]} \norm{u_n(t)}_\EA^{r}\Big]<\infty
\end{align}
for all $r\in [1,\infty).$ Using several compactness Lemmata for spaces of \cadlag functions inspired by \cite{Motyl} and  \cite{Brz+Manna_2017_SLLGELevy}, \eqref{energyEstimateIntroduction} leads to tightness of the sequence $\left(u_n\right)_{n\in\N}$ in
\begin{align*}
	Z_T&:=\cadlagHminuseins\cap \LalphaPlusEinsAlphaPlusEins \cap \weaklyCadlag{\EA}.
\end{align*}
For the precise definition of $Z_T,$ we refer to Section \ref{CompactnessSection}. Subsequently, a limit argument based on the Skorohod-Jakubowski Theorem shows the existence of a martingale solution. Altogether, we prove the following result.


\begin{theorem}\label{mainTheorem}
	Choose the operator $A$ and the energy space $\EA$ according to Assumption \ref{spaceAssumptions}, the nonlinearity $F$ according to Assumptions \ref{nonlinearAssumptions} and \ref{focusing} and the noise according to Assumption \ref{noiseAssumptions}.
	Then, for any $u_0\in \Lzwei,$ the problem \eqref{ProblemMarcus}
	has a martingale solution $\left(\tilde{\Omega},\tilde{\F},\tilde{\Prob},\tilde{\eta},\tilde{\Filtration},u\right)$ which satisfies
	\begin{align*}
	u\in L^q(\tilde{\Omega},L^\infty(0,T;\EA))
	\end{align*}
	for all $q\in [1,\infty).$ Moreover,
	the  equality 
	\begin{align}\label{eqn-normPreservation}
	\norm{ u(t)}_\Lzwei  & =  \norm{ u_0}_\Lzwei
	\end{align}
	holds $\bar{\Prob}$-almost surely for all $t\in [0,T].$
\end{theorem}

The article is organized as follows. In the second section, we fix the setting by stating the general assumptions on the operator $A,$ the nonlinearity $F$ and  the noise term. These assumptions are illustrated in the third section by concrete examples. The proof of the main Theorem \ref{mainTheorem} is contained in the sections 4, 5 and 6 that deal with compactness results, the uniform estimates for the Galerkin approximation and the limit procedure. In the appendix, we collect basic material on Poisson random measures and Marcus noise.  
\section{General Framework and Assumptions}

	In this section, we formulate the abstract framework for the stochastic nonlinear Schr\"{o}dinger equation we refer to in Theorem \ref{mainTheorem}.

	Let $(\tilde{M},\Sigma,\mu)$ be a $\sigma$-finite measure space with metric $\rho$ satisfying the \emph{doubling property}, i.e. $\mu(B(x,r))<\infty$ for all $x\in \tilde{M}$ and $r>0$ and
	\begin{align}\label{doubling}
	\mu(B(x,2r))\lesssim \mu(B(x,r)).
	\end{align}	
	Let $M\subset \tilde{M}$ be an open subset with finite measure and  $L^q(M)$ for $q\in [1,\infty]$ the space of equivalence classes of $\C$-valued $q-$integrable functions.
	We further abbreviate ${H}:=L^2(M)$ and equip $H$ with the standard complex $L^2$-inner product.
	
	Let $A$ be a non-negative self-adjoint operator on ${H}$ with domain $\mathcal{D}(A).$
	We set  $\EA:=\mathcal{D}((\Id+A)^\frac{1}{2})$ and call it \emph{energy space}. Equipped  with the inner product
	\begin{align*}
	\skp{x}{y}_\EA=\skpLzwei{(\Id+A)^\frac{1}{2} x}{(\Id+A)^\frac{1}{2} y},
	\end{align*}
	$\EA$ is a complex Hilbert space.
	Moreover,
	we define the extrapolation space $H_{-\frac{1}{2}}$ as the completion of $\Lzwei$ with respect to the norm
	\begin{align*}
	\norm{x}_{-\frac{1}{2}}:=\norm{(\Id+A)^{-\frac{1}{2}}x}_{L^2},\qquad x\in \Lzwei,
	\end{align*}
	and obtain a Hilbert space with the inner product
	\begin{align*}
	\skp{x}{y}_{-\frac{1}{2}}&=\lim_{n,m\to\infty} \skpLzwei{(\Id+A)^{-\frac{1}{2}} x_n}{(\Id+A)^{-\frac{1}{2}}y_m},\qquad x,y\in H_{-\frac{1}{2}},
	\end{align*}
	for sequences $\left(x_n\right)_{n\in\N},\left(y_m\right)_{m\in\N}\subset \Lzwei$ with $x_n\to x$ and $y_m\to y$ in $H_{-\frac{1}{2}}$ as $n,m\to \infty.$
	Note that we can identify $H_{-\frac{1}{2}}$ with $\EAdual$ and the duality is given by
	\begin{align*}
	\duality{x}{y}_{\frac{1}{2},-\frac{1}{2}}:=\lim_{n\to\infty}\skpLzwei{x}{y_n},\qquad x\in \EA, \quad y\in H_{-\frac{1}{2}},
	\end{align*}
	with $\left(y_n\right)_{n\in\N}\subset \Lzwei$ such that $y_n\to y$ in $H_{-\frac{1}{2}}$ as $n\to\infty.$ Often, we shortly write $\duality{\cdot}{\cdot}$ for $\duality{\cdot}{\cdot}_{\frac{1}{2},-\frac{1}{2}}$ and write $\EAdual$ instead of $H_{-\frac{1}{2}}.$
		 Note that $\left(E_A, H, E_A^*\right)$ is a Gelfand triple, i.e.
		 \begin{align*}
		 E_A\hookrightarrow H \cong H^* \hookrightarrow E_A^*.
		 \end{align*}
	We point out that one can also treat $H,$ $\EA$ and $H_{-\frac{1}{2}}$ as real Hilbert spaces with scalar products $\Real\skpH{\cdot}{\cdot},$ $\Real\skp{\cdot}{\cdot}_\EA$ and $\Real\skp{\cdot}{\cdot}_{-\frac{1}{2}},$ respectively. Then, $\EA$ and $H_{-\frac{1}{2}}$ are dual in the sense that each real-valued continuous linear functional $f$ on $\EA$ has the representation $f=\Real \duality{\cdot}{y_f}_{\frac{1}{2},-\frac{1}{2}}$ for some $y_f\in H_{-\frac{1}{2}}.$\\

We continue with the main Assumption on the functional analytic setting for the stochastic NLS.

	\begin{assumptionNotation}\label{spaceAssumptions}
		We assume the following:
		\begin{itemize}
			\item [i)] There is a strictly positive self-adjoint operator $S$ on ${H}$ with compact resolvent commuting with $A$ and $\mathcal{D}(S^k)\hookrightarrow \EA$ for some $k\in \N.$ Moreover, we assume that there exists $p_0\in [1,2),$ such that $S$  has \emph{generalized Gaussian $(p_0,p_0^\prime)$-bounds}, i.e.
			\begin{align}\label{generalizedGaussianEstimate}
			\Vert \mathbf{1}_{B(x,t^\frac{1}{m})}e^{-tS}\mathbf{1}_{B(y,t^\frac{1}{m})}\Vert_{\mathcal{L}(L^{p_0},L^{p_0^\prime})} \le C{\mu(B(x,t^\frac{1}{m}))}^{\frac{1}{p_0^\prime}-\frac{1}{p_0}} \exp \left\{-c \left(\frac{\rho(x,y)^m}{t}\right)^{\frac{1}{m-1}}\right\},
			\end{align}
			for all $t>0$ and  $(x,y)\in M\times M$ with constants $c,C>0$ and $m\ge 2.$
						\item[ii)] Let $\alpha \in (1,p_0^\prime-1)$ be such that $\EA$ is compactly embedded in $\LalphaPlusEins.$ We set
						\begin{align*}
						p_{\max}:= \sup \left\{p\in (1,\infty]: \EA \hookrightarrow L^p(M) \quad\text{is continuous}\right\}\dela{\in [\alpha+1,\infty].}
						\end{align*}
						and note that $p_{\max}\in [\alpha+1,\infty].$
						In the case $p_{\max}<\infty,$ we assume that $\EA \hookrightarrow L^{p_{\max}}(M)$ is continuous, but not necessarily compact.
		\end{itemize}
	\end{assumptionNotation}
	
	\begin{remark}\label{GaussianRemark}
		\begin{itemize}
			\item[a)] If $p_0=1,$ then it is proved in \cite{blunckKunstmann} that \eqref{generalizedGaussianEstimate} is equivalent to the usual upper Gaussian estimate, i.e. for all $t>0$ there is a measurable function $p(t,\cdot,\cdot): M\times M\to \R$ with
			\begin{align*}
			(e^{-tS}f)(x)= \int_M p(t,x,y) f(y) \mu(dy), \quad t> 0, \quad \text{a.e. } x\in M
			\end{align*}
			for all $f\in H$ and
			\begin{align}\label{GaussianEstimate}
			\vert p(t,x,y)\vert \le \frac{C}{\mu(B(x,t^\frac{1}{m}))} \exp \left\{-c \left(\frac{\rho(x,y)^m}{t}\right)^{\frac{1}{m-1}}\right\},
			\end{align}
			for all $t>0$ and almost all $(x,y)\in M\times M$ with constants $c,C>0$ and $m\ge 2.$	In particular, $e^{-tS}$ can be extended to a $C_0$-semigroup on $L^p(M)$ for all $p\in [1,\infty).$
			\item[b)] In fact, in all our examples in the third section, the upper Gaussian estimate \eqref{GaussianEstimate} holds and therefore, the previous assumption is fulfilled with $p_0=1.$
		\end{itemize}
	\end{remark}
	The following Lemma contains some straightforward consequences of \ref{spaceAssumptions}.
	\begin{lemma}\label{spaceLemma}
		\begin{itemize}
			\item[a)] 		There is a positive self-adjoint operator $\hat{A}$ on $E_A^*$ with $\mathcal{D}(\hat{A})=E_A$ such that the restriction of  $\hat{A}$ to $D(A)$ is equal to $A$. \dela{on $H.$} 	For simplicity of notation, we will \del{also} denote the operator  $\hat{A}$ by $A.$
			\item[b)]  The embedding
			$\EA \hookrightarrow {H}$
			is compact.
			\item[c)] There is an orthonormal basis $\left(h_n\right)_{n\in \N}$ and a nondecreasing sequence $\left(\lambda_n\right)_{n\in\N}$ with $\lambda_n>0$ and  $\lambda_n\to \infty$ as $n\to \infty$ and
			\begin{align*}
			S x=\sum_{n=1}^\infty \lambda_n \skpH{x}{h_n} h_n, \quad x\in \mathcal{D}(S)=\left\{x\in H: \sum_{n=1}^\infty \lambda_n^2 \vert \skpH{x}{h_n}\vert^2<\infty\right\}.
			\end{align*}
		\end{itemize}
	\end{lemma}

	\begin{assumption}\label{nonlinearAssumptions}
		Let $\alpha\in(1,p_0^\prime-1)$ be chosen as in Assumption $\ref{spaceAssumptions}.$ Then, we assume the following:
		\begin{itemize}
			\item[i)] Let $F: \LalphaPlusEins \to \LalphaPlusEinsDual$ be a function satisfying the following estimate
			\begin{align}\label{nonlinearityEstimate}
			\norm{F(u)}_\LalphaPlusEinsDual \le C_{F,1} \norm{u}_\LalphaPlusEins^\alpha,\quad u\in \LalphaPlusEins.
			\end{align}
			Note that this leads to $F: \EA \to \EAdual$ by Assumption $\ref{spaceAssumptions},$ because $\EA\hookrightarrow\LalphaPlusEins$ implies $(\LalphaPlusEins)^*=\LalphaPlusEinsDual\hookrightarrow \EAdual.$ We further assume $F(0)=0$ and
			\begin{align}\label{nonlinearityComplex}
			\Real \duality{\im u}{F(u)}=0, \quad u\in \LalphaPlusEins.
			\end{align}
			\item[ii)] The map $F: \LalphaPlusEins\to \LalphaPlusEinsDual$ is continuously real Fr\'{e}chet differentiable with
			\begin{align}\label{deriveNonlinearBound}
			\Vert F'[u]\Vert_{L^{\alpha+1}\to L^\frac{\alpha+1}{\alpha}} \le C_{F,2} \norm{u}_\LalphaPlusEins^{\alpha-1}, \quad u\in \LalphaPlusEins.
			\end{align}
			\item[iii)] The map $F$ has a real antiderivative $\hat{F},$ i.e. there exists a Fr\'{e}chet-differentiable map  $\hat{F}: \LalphaPlusEins\to \R$ with
			\begin{align}\label{antiderivative}
			\Fhat'[u]h=\Real \duality{F(u)}{h},\quad u,h\in \LalphaPlusEins.
			\end{align}
		\end{itemize}
	\end{assumption}
	By Assumption $\ref{nonlinearAssumptions}$ ii) and the mean value theorem, we get
	\begin{align}\label{nonlinearityLocallyLipschitz}
	\norm{  F(x)-F(y)}_{\LalphaPlusEinsDual}&\le\sup_{t\in [0,1]}\norm{F'[tx+(1-t)y]}	\norm{  x-y}_{\LalphaPlusEins}\nonumber\\
	&\le C_{F,2} \left(\norm{x}_\LalphaPlusEins+\norm{y}_\LalphaPlusEins\right)^{\alpha-1} \norm{x-y}_\LalphaPlusEins	
	\end{align}	
	for $x,y\in \LalphaPlusEins$ which means that the nonlinearity is  Lipschitz on bounded sets of $\LalphaPlusEins.$
We will cover the following two standard types of nonlinearities.
	\begin{definition}
		Let $F$ satisfy Assumption $\ref{nonlinearAssumptions}.$
		Then, $F$ is called \begin{trivlist}
\item[\pb]  \emph{defocusing},
 if $\Fhat(u)\ge 0$ for all $u\in \LalphaPlusEins$
 \item[\pb]  and
 \item[\pb] \emph{focusing}, if $\Fhat(u)\le 0$ for all $u\in \LalphaPlusEins.$	
 \end{trivlist}
	\end{definition}
	

	\begin{Assumption}\label{focusing}
		We assume that either condition i) or condition i') holds, where
		\begin{itemize}
			\item[i)] The function $F$ is defocusing and satisfies
			\begin{align}\label{boundantiderivative}
			\frac{1}{C_{F,3}}\norm{u}_\LalphaPlusEins^{\alpha+1}\le  \Fhat(u)\le C_{F,3} \norm{u}_\LalphaPlusEins^{\alpha+1}, \quad u\in \LalphaPlusEins.
			\end{align}
			\item[i')] The function  $F$ is focusing and satisfies
			\begin{align}\label{boundantiderivativeFocusing1}
			-\Fhat(u)\le C_{F,4}\norm{u}_\LalphaPlusEins^{\alpha+1}, \quad u\in \LalphaPlusEins,
			\end{align}
			and there exists  $\theta \in (0,\frac{2}{\alpha+1})$ such that \footnote{In below, the symbol $\left(\cdot, \cdot\right)_{\theta,1}$ stands for the real interpolation functor with parameters $1$ and $\infty$, see for instance \cite{Triebel}.}
			\begin{align}\label{interpolationFocusing}
			\left({H},\EA\right)_{\theta,1}\hookrightarrow \LalphaPlusEins.
			\end{align}
		\end{itemize}
	\end{Assumption}
	The model nonlinearities are the defocusing power nonlinearity $F_{\alpha}^+(u):=\vert u\vert^{\alpha-1}u$ with subcritical exponents in the sense that the embedding $\EA \hookrightarrow L^{\alpha+1}$ is compact and the focusing  nonlinearity $F_{\alpha}^-(u):=-\vert u\vert^{\alpha-1}u$ with an additional restriction to the power $\alpha.$
	\begin{assumption}\label{noiseAssumptions}
		\begin{itemize}
			\item[(a)]  Assume that $\big(\Omega, \mathcal{F}, \mathbb{F}, \mathbb{P}\big)$ is a filtered probability space, where $\mathbb{F}=\big(\mathcal{F}_t)_{t\geq 0}$ is the filtration, and this probability space satisfies the so called usual conditions, i.e.
			\begin{trivlist}
				\item[(i)] $\mathbb{P}$ is complete on $(\Omega, \mathcal{F})$,
				\item[(ii)] for each $t\geq 0$, $\mathcal{F}_t$ contains all $(\mathcal{F},\mathbb{P})$-null sets,
				\item[(iii)] the filtration $\mathbb{F}$ is right-continuous.
			\end{trivlist}
			\item[(b)] Assume that $(L(t))_{t\geq 0}$ is an $\R^N$-valued, $(\mathcal{F}_t)$-adapted L\'evy process of  pure jump type defined on the above probability space with drift $0$ and the corresponding time homogenous Poisson random measure $\eta$.
			\item[(c)]    Assume that the intensity measure $\Leb \otimes\nu$ is such that $\supp \nu \subset B$, where $B$ is the closed unit ball in
			$\R^N$.
\item[d)] Let $B_1,\dots, B_M\in \LinearOperators{H}$ be self-adjoint operators on $H$  with $B_m|_\EA\in {\LinearOperators{\EA}}$ and $B_m|_\LalphaPlusEins \in {\LinearOperators{\LalphaPlusEins}}.$
		\end{itemize}
	\end{assumption}
	We abbreviate
\begin{align}
b_\EA:=\sumM \|B_m\|_\LinearOperators{\EA}^2,\qquad b_{L^{\alpha+1}}:=\sumM \|B_m\|_\LinearOperators{L^{\alpha+1}}^2,\qquad b_\Lzwei:=\sumM \|B_m\|_\LinearOperators{\Lzwei}^2
\end{align}
and for $l\in \R^N,$ we introduce the notation
\begin{align*}
\B(l)=:\sumM l_m B_m.
\end{align*}
\begin{remark}\label{LevyKhintchine}
	Note that  by the L\'evy-Khinchine formula, see \cite{PeszatZabczyk}, Theorem 4.23, the previous assumption yields that the intensity measure $\nu$ is a L\'evy-measure on $\R^N,$ i.e.
	\begin{align}\label{LevyMeasure}
	\int_B \vert l\vert^2 \nu(\df l)<\infty.
	\end{align}
Moreover, we have the representation
\begin{equation*}
L(t) = \int_{0}^{t}\int_{B}\!l \,\tilde{\eta}(\df s,\df l). 
\end{equation*}		
\end{remark}

	\subsection{The Marcus Mapping}\label{subsec-Marcus}
	

	Let us define a generalized Marcus mapping
	\begin{align*}
		\Phi: \mathbb{R}_{+}  \times \R^N \times H \rightarrow H,\qquad \Phi(t,l,x):=e^{-\im t\B(l)}x,
	\end{align*}
	i.e. for each fixed $l \in \R^N$, $x \in H$, the function $t \mapsto \Phi(t,l,x)$ is the continuously differentiable solution of
	\begin{equation}\label{E_Phieqn}
	\dfrac{du}{dt}(t) = -\im\sum_{m=1}^{N} l_m B_mu(t),\qquad t\geq 0,
	\end{equation} with $u(0)=x \in H$, and $l=(l_1, l_2, \ldots, l_N)\in \R^N$.
	Equation \eqref{ProblemMarcus} with notation $\diamond$ is defined in the integral form as following
	\begin{align}\label{marcus2}
	u(t)  & =  u_0 -\im \int_0^t\left( A u(s)+F(u(s))\right) \,\df s + \int_{0}^{t}\! \int_{B} \!\left[\groupB u(s-) - u(s-)\right]\, \tilde{\eta}(\df s,\df l)\nonumber\\
	&\quad + \int_{0}^{t}\! \int_{B} \! \left\{\groupB u(s) - u(s) + \im \sumM l_mB_mu(s)\right\}\, \nu(\df l)\df s,
	\end{align}
	where $\tilde{\eta}:=\eta-\Leb \otimes\nu$ denotes the compensated Poisson random measure induced by $\eta.$ In the next definition, we define the notion of a solution used in the present article.


	\begin{definition}\label{MartingaleSolutionDef}
		Let $T>0$ and $u_0\in E_A.$ A \emph{martingale solution} of the equation $\eqref{ProblemMarcus}$ is a system $\left(\bar{\Omega},\bar{\F},\bar{\Prob},\bar{\eta},\bar{\Filtration},u\right)$ with
		\begin{itemize}
			\item a probability space $\left(\bar{\Omega},\bar{\F},\bar{\Prob}\right);$
			\item a time homogeneous Poisson random measure $\bar{\eta}$  on $\R^N$ over $\bar{\Omega}$ with intensity measure $\nu,$
			\item  a filtration $\bar{\Filtration}=\left(\bar{\F}_t\right)_{t\in [0,T]}$ with the usual conditions;
			\item an  $\bar{\Filtration}$-adapted, $\EAdual$-valued c\`adl\`ag process  such that $u\in L^2(\Omega\times [0,T],\EAdual)$
			and almost all paths are in $\D_w([0,T],\EA)$,
		\end{itemize}
		such that 
the equation \eqref{marcus2}
		holds $\bar{\Prob}$-almost surely in $\EAdual$ for all $t\in [0,T]$
		with $\tilde{\bar{\eta}}$ instead of $\tilde{\eta}.$
	\end{definition}
	
\section{Examples}

In this section, we collect concrete settings which are covered by the general framework of Assumptions \ref{spaceAssumptions}, \ref{nonlinearAssumptions} and \ref{focusing}. We skip the proofs since they already appeared in \cite{ExistencePaper}, where the NLS with Gaussian noise was considered in the same framework.

\begin{corollary}\label{CorollaryNLS}
	Suppose that a) or b) or c) is true.
	\begin{itemize}
		\item[a)] Let $M$ be a d-dimensional compact manifold, $A=-\Delta_g,$ $E_A=H^1(M).$
		\item[b)] Let $M\subset \Rd$ be a bounded domain and $A=-\Delta_D$ be the Dirichlet-Laplacian, $E_A=H^1_0(M).$
		\item[c)] Let $M\subset \Rd$ be a bounded Lipschitz domain and $A=-\Delta_N$ be the Neumann-Laplacian, $E_A=H^1(M).$
	\end{itemize}	
	Choose the nonlinearity from $i)$ or $ii).$
	\begin{itemize}
		\item[i)] $F(u)= \vert u\vert^{\alpha-1}u$ with $ \alpha \in \left(1,1+\frac{4}{(d-2)_+}\right)$, i.e. $F$ is defocusing, 
		\item[ii)]$F(u)= -\vert u\vert^{\alpha-1}u$ with $ \alpha \in \left(1,1+\frac{4}{d}\right),$ i.e. $F$ is focusing, 
	\end{itemize}
	Set $B_m x=e_m x$ for $x\in H$ and $m=1,\dots,M$ with real-valued functions
		\begin{align}\label{assumptionEm}
		e_m \in F:=\begin{cases}
		H^{1,d}(M) \cap \LInfty, &  d\ge 3,\\
		H^{1,q}(M),& d=2,\\
		H^{1}(M),& d=1,\\
		\end{cases}
		\end{align}
		for some $q>2$ in the case $d=2.$
	Then, the problem
\begin{equation}
\label{ProblemMarcusExample}
\left\{
\begin{aligned}
\d u(t)&= \left(-\im A u(t)-\im  F(u(t))\right) dt-\im \sumM B_mu(t) \diamond dL_m(t) ,\\
u(0)&=u_0\in \EA,
\end{aligned}\right.
\end{equation}
		has a martingale solution which satisfies $\norm{u(t)}_H=\norm{u_0}_H$ almost surely for all $t\in[0,T]$ and
		\begin{align*}
		u\in L^q(\tilde{\Omega},L^\infty(0,T;\EA))
		\end{align*}
		for all $q\in [1,\infty).$
\end{corollary}

\begin{proof}
	We refer to \cite{ExistencePaper}, Section 3, for the verification of the Assumptions in Theorem \ref{mainTheorem}.
\end{proof}	

Additionally to the stochastic NLS, we can also cover the fractional NLS with the Laplacians replaced by their fractional powers.

\begin{corollary}\label{CorollaryFractionalNLS}
	Choose one of the settings a), b) or c) in Corollary. Let $\beta>0$ and suppose that we have either $i)$ or $ii)$ below.
		\begin{itemize}
			\item[i)] $F(u)= \vert u\vert^{\alpha-1}u$ with $ \alpha \in \left(1,1+\frac{4\beta}{(d-2\beta)_+}\right)$,
			\item[ii)]$F(u)= -\vert u\vert^{\alpha-1}u$ with $ \alpha \in \left(1,1+\frac{4\beta}{d}\right),$
		\end{itemize}
	Let $B_m$ for $m=1,\dots, M$ as in Assumption \ref{noiseAssumptions}.
		Then, the problem
		\begin{equation}
		\label{ProblemMarcusExample}
		\left\{
		\begin{aligned}
		\d u(t)&= \left(-\im A^\beta u(t)-\im  F(u(t))\right) dt-\im \sumM B_mu(t) \diamond dL_m(t) ,\\
		u(0)&=u_0\in \mathcal{D}(A^\frac{\beta}{2}),
		\end{aligned}\right.
		\end{equation}
		has a martingale solution which satisfies $\norm{u(t)}_H=\norm{u_0}_H$ almost surely for all $t\in[0,T]$ and 
		\begin{align*}
		u\in L^q(\tilde{\Omega},L^\infty(0,T;\mathcal{D}(A^\frac{\beta}{2})))
		\end{align*}
		for all $q\in [1,\infty).$
\end{corollary}

\section{Compactness and Tightness Criteria}	\label{CompactnessSection}

This section is devoted to the compactness results which will be used to get  a martingale solution of $\eqref{ProblemMarcus}$ by the Faedo-Galerkin method.
We begin with a definition of the \cadlag functions and a generalization of the modulus of continuity to this class.	Throughout the section, $(\mathbb{S},\df)$ denotes a complete, separable metric space.

\begin{definition}\label{defCadlag}
	\begin{itemize}
		\item[a)] 		
		The space of all \emph{\cadlag functions} $f: [0,T]\to \mathbb{S},$ i.e. $f$ is right-continuous with left limit in every $t\in [0,T],$ is called $\D([0,T],\mathbb{S}).$
		\item[b)] For $u\in \D([0,T],\mathbb{S})$ and $\delta>0,$ we define the \emph{modulus}
		\begin{align*}
		w_\mathbb{S}(u,\delta):=\inf_{\Pi_\delta} \max_{t_j\in Q} \sup_{t,s\in [t_{j-1},t_j)}\df(u(t),u(s)),
		\end{align*}
		where $\Pi_\delta$ is the set of all partitions $Q=\left\{0=t_0<t_1<\dots<t_N=T\right\}$ of $[0,T]$ with
		\begin{align*}
		t_{j+1}-t_j\ge \delta,\qquad j=0,\dots, N-1.
		\end{align*}
		\item[c)] 	We denote the set of increasing homeomorphisms of $[0,T]$ by $\Lambda$ and  we equip $\D([0,T],\mathbb{S})$ with the metric defined by
			\begin{align*}
			\rho(u,v):= \inf_{\lambda \in \Lambda}\left[\sup_{t\in [0,T]}\df(u(t),v(\lambda(t)))+\sup_{t\in [0,T]} \vert t-\lambda(t)\vert+\sup_{s\neq t}\left\vert \log\frac{\lambda(t)-\lambda(s)}{t-s}\right\vert\right]
			\end{align*}
			for $u,v \in \D([0,T],\mathbb{S}).$	
	\end{itemize}
\end{definition}

The following Proposition is about the so-called \emph{Skohorod-topology} on $\D([0,T],\mathbb{S}).$

\begin{Prop}\label{cadlagTopology}
	\begin{itemize}
			\item[a)]  The pair $\bigl(\D([0,T],\mathbb{S}),\rho \bigr)$ is  a complete, separable metric space.
			\item[b)] A sequence $\left(u_n\right)_{n\in\N}\in \D([0,T],\mathbb{S})^\N$ is convergent to $u\in \D([0,T],\mathbb{S})$ in the metric $\rho$ if and only if
			there exists $\left(\lambda_n\right)_{n\in\N}\in \Lambda^\N$ with
			\begin{align*}
			\sup_{t\in [0,T]} \vert \lambda_n(t)-t\vert \to 0,\qquad
			\sup_{t\in [0,T]} \df(u_n(\lambda_n(t)),u(t))\to 0,\qquad n\to \infty.
			\end{align*}
	\end{itemize}
\end{Prop}

\begin{proof}
	See \cite{Billingsley}, page 123 and following for a proof.
\end{proof}

\begin{definition}\label{DefinitionWeakTopologySpaces}
	Let $\mathbb{K}\in\{\R,\C\}$ and let  $X$ be a reflexive, separable $\mathbb{K}$-Banach space and $X^*$ its dual.
	\begin{enumerate}
		\item[a)] 	Then, we define
		$\weaklyCadlag{X}$ as the space of all  $u: [0,T]\to X$ such that
		\begin{align*}
			[0,T]\ni t\to \duality{u(t)}{x^*}\in \mathbb{K} \text{ is \cadlag for all $x^*\in X^*.$}
		\end{align*}
		We equip $\weaklyCadlag{X}$ with the weakest topology such that the map
		\begin{align*}
		\weaklyCadlag{X} \ni u \mapsto \duality{u(\cdot)}{x^*}\in \D([0,T],\mathbb{K})
		\end{align*}
		is continuous for all $x^*\in X^*.$
		\item[b)] For $r>0,$ we consider the ball
		$\mathbb{B}_X^r:=\left\{u\in X: \norm{u}_X\le r \right\}$ and define
		\begin{align*}
		\cadlagBallX:= \left\{ u\in \weaklyCadlag{X}: \sup_{t\in [0,T]} \norm{u(t)}_X \le r \right\}.
		\end{align*}
	\end{enumerate}
\end{definition}

\begin{remark}
	By the separability of $X,$  the weak topology on $\mathbb{B}_X^r$ is metrizable and we choose a corresponding metric $q.$ The notation in Definition \ref{DefinitionWeakTopologySpaces} is justified, i.e.
	\begin{align}\label{assertionRemarkWeakTop}
		\text{$\cadlagBallX$ coincides with $\D([0,T],\mathbb{S})$ for $(\mathbb{S},d)=(\mathbb{B}_X^r,q).$}
	\end{align}
	  In particular, $\cadlagBallX$ is a complete, separable metric space by Proposition  \ref{cadlagTopology}.	
	To show \eqref{assertionRemarkWeakTop}, we note that the right-continuity of $\duality{u(\cdot)}{x^*}$ for all $x^*\in X^*$ is equivalent to the right-continuity of $u$ in $(\mathbb{B}_X^r,q)$ by the definition of $q$. It is also easy to see that the existence of left limits transfers from $(\mathbb{B}_X^r,q)$ to $ \duality{\cdot}{x^*}$ for all $x^*\in X^*.$
	
	 For the converse direction, let $t_n\to t-.$ Then, for each $x^*\in X^*,$ there is $\gamma_{x^*}\in \mathbb{K}$ with $\duality{u(t_n)}{x^*}\to \gamma_{x^*}.$ Since $X$ is reflexive, $x^*\mapsto \gamma_{x^*}$ is linear and $\vert \gamma_{x^*}\vert \le r \norm{x^*}_{X^*},$ there is $v\in X$ such that $\gamma_{x^*}=\duality{v}{x^{*}}.$ Hence,  $q(u(t_n),v)\to 0.$
\end{remark}

\begin{lemma}\label{lem-boundednessWeaklyCadlag}
		Let $\mathbb{K}\in\{\R,\C\}$ and let  $X$ be a reflexive, separable $\mathbb{K}$-Banach space and let $u_n,u\in \weaklyCadlag{X}$ with $u_n\to u$ in $\weaklyCadlag{X}$ as $n\to\infty.$ 
	Then, we have
		\begin{align*}
	\sup_{n\in\N}\sup_{t\in [0,T]}\norm{u_n(t)}_X<\infty.
	\end{align*} 
\end{lemma}

\begin{proof}
	From $u_n\to u$ in $\weaklyCadlag{X}$ as $n\to\infty,$ we infer that for every $x^*\in X^*,$ we have $\duality{u_n}{x^*}\to \duality{u}{x^*}$ in $\D([0,T],\mathbb{K})$ as $n\to\infty.$ Proposition \ref{cadlagTopology} therefore implies that for every $x^*\in X^*,$ there exists $\left(\lambda_n\right)\in \Lambda^\N$ with
	\begin{align*}
	\sup_{t\in [0,T]} \vert \lambda_n(t)-t\vert \to 0,\qquad
	\sup_{t\in [0,T]} \vert \duality{u_n(\lambda_n(t))}{x^*}-\duality{u(t)}{x^*}\vert \to 0,\qquad n\to \infty.
	\end{align*}
	In particular, we obtain
	\begin{align*}
		\sup_{n\in\N}\sup_{t\in [0,T]} \vert \duality{u_n(t)}{x^*}\vert=\sup_{n\in\N}\sup_{t\in [0,T]} \vert \duality{u_n(\lambda_n(t))}{x^*}\vert<\infty
	\end{align*}
	for every $x^*\in X^*.$ The uniform boundedness principle yields 
	\begin{align*}
		\sup_{n\in\N}\sup_{t\in [0,T]}\norm{u_n(t)}_X =\sup_{n\in\N}\sup_{t\in [0,T]}\sup_{\norm{x^*}_{X^*}\le 1}\vert \duality{u_n(t)}{x^*}\vert<\infty.
	\end{align*}
\end{proof}

We recall that the energy space $\EA$ is defined by $\EA:={\mathcal{D}((\Id+A)^\frac{1}{2})}.$
We continue with a criterion for convergence of a sequence in $\cadlagBall.$

\begin{lemma}\label{convergenceStetigBall}
	Let $r>0$ and $u_n: [0,T]\to \EA$  functions such that
	\begin{enumerate}
		\item[a)] $\sup_{n\in\N} \sup_{s\in [0,T]}\norm{u_n(s)}_\EA \le r$,
		\item[b)] $u_n\to u$ in $\cadlagHminuseins$ for $n\to \infty.$
	\end{enumerate}
	Then $u_n,u\in \cadlagBall$ for all $n\in\N$ and $u_n \to u$ in $\cadlagBall$ for $n\to \infty.$
\end{lemma}

\begin{proof}
	See \cite{Motyl}, Lemma 3.3.
\end{proof}

We continue with a Lemma stated in Lions \cite{Lions}, p. 58.

\begin{lemma}[Lions]\label{LionsLemma}
	Let $X,X_0,X_1$ be Banach spaces with $X_0 \hookrightarrow X \hookrightarrow X_1$ where the first embedding is compact. Assume furthermore that $X_0, X_1$ are reflexive and $p\in [1,\infty).$ Then, for each $\varepsilon>0$ there is $C_\varepsilon>0$ with
	\begin{align*}
	\Vert x \Vert_X^p \le \varepsilon \Vert x\Vert_{X_0}^p+C_\varepsilon \Vert x\Vert_{X_1}^p,\quad x\in X_0.
	\end{align*}
\end{lemma}

\begin{proof}
	See \cite{FHornungPhD}, Lemma 2.34.
\end{proof}

We define a space $Z_T$ by
\begin{align}
\label{eqn-Z_T}
Z_T&:=\cadlagHminuseins\cap \LalphaPlusEinsAlphaPlusEins \cap \weaklyCadlag{\EA}=:Z_1\cap Z_2\cap Z_3.
\end{align}
We equip $Z_T$  with the supremum-topology, i.e. the smallest topology that contains $\bigcup_{j=1}^3 \mathcal{O}_j,$ where $\mathcal{O}_j$ is the trace of the $Z_j$-topology in $Z_T.$

In the next Proposition, we give a criterion for compactness in $Z_T.$ This result generalises Theorem 2 of Section 3 from \cite{Motyl}. For a continuous version of this result see Proposition 4.2 of \cite{ExistencePaper}. Our proof is along the similar lines to Proposition 5.7 of the first and third named authours \cite{Brz+Manna_2017_SLLGELevy}.

\begin{Prop}\label{CompactnessDeterministic}
	Let $K$ be a subset of $Z_T$ and $r>0$ such that
	\begin{enumerate}
		\item[a)] $
		\sup_{z\in K} \sup_{t\in [0,T]}\norm{z(t)}_\EA\le r ;
		$
		\item[b)] $
		\lim_{\delta \to 0} \sup_{z\in K} w_{\EAdual}(z,\delta)=0.$
	\end{enumerate}
	Then, $K$ is relatively compact in $Z_T.$
\end{Prop}

\begin{proof}
	Let $K$ be a subset of $Z_T$ such that the assumptions $a)$ and $b)$ are fullfilled and $\left(z_n\right)_{n\in\N}\subset K.$
	
	\emph{Step 1:} The relative compactness of $K$ in $\cadlagHminuseins$ is an immediate consequence of Theorem 3.2 in \cite{Motyl}. Hence, we can take a subsequence again denoted by $\left(z_n\right)_{n\in\N}$ and $z\in \cadlagHminuseins$ with $z_n \to z$ in $\cadlagHminuseins.$ By Lemma \ref{convergenceStetigBall}, we infer that  $z_n \to z$ in $\weaklyCadlag{\EA}$ and $\sup_{t\in [0,T]}\norm{z(t)}_\EA\allowbreak\le r.$

	\emph{Step 2:} We fix again $\varepsilon>0.$ By \dela{Lions'}{the Lions} Lemma $\ref{LionsLemma}$ with  $X_0=\EA,$ $X=\LalphaPlusEins,$ \\$X_1=\EAdual,$ $p=\alpha+1$ and $\varepsilon_0=\frac{\varepsilon}{2 T (2r)^{\alpha+1}}$ we get
	\begin{align}\label{LionsEstimate}
	\Vert v \Vert_\LalphaPlusEins^{\alpha+1} \le \varepsilon_0 \Vert v\Vert_{\EA}^{\alpha+1}+C_{\varepsilon_0} \Vert v\Vert_{\EAdual}^{\alpha+1}
	\end{align}
	for all $v\in\EA.$
	Integration with respect to time yields
	\begin{align*}
	\Vert z_n-z\Vert_\LalphaPlusEinsAlphaPlusEins^{\alpha+1} &\le {\varepsilon_0} \Vert z_n-z\Vert_{L^{\alpha+1}(0,T;\EA)}^{\alpha+1}+C_{\varepsilon_0} \Vert z_n-z\Vert_{L^{\alpha+1}(0,T;\EAdual)}^{\alpha+1};
	\end{align*}
	\begin{align*}
		{\varepsilon_0} \Vert z_n-z\Vert_{L^{\alpha+1}(0,T;\EA)}^{\alpha+1}
		\le  {\varepsilon_0} T \Vert z_n-z\Vert_{L^\infty(0,T;\EA)}^{\alpha+1}
		\le {\varepsilon_0} T \left(2 r\right)^{\alpha+1}
		\le \frac{\varepsilon}{2}.
	\end{align*}
	By \cite{Billingsley}, p.124, equation $(12.14),$ convergence in $\cadlagHminuseins$ implies $z_n(t)\to u(t)$ in $\EAdual$ for almost all $t\in [0,T].$ By Assumption $a),$ Lebesgue's Theorem yields $z_n \to z$ in $L^{\alpha+1}(0,T;\EAdual).$ Hence,
	\begin{align*}
		\limsup_{n\to \infty} \Vert z_n-z\Vert_\LalphaPlusEinsAlphaPlusEins^{\alpha+1}\le \frac{\varepsilon}{2}
	\end{align*}
	for all $\varepsilon>0$ and thus, the sequence $\left(z_n\right)_{n\in\N}$ is also converges to $u$ in $\LalphaPlusEinsAlphaPlusEins$.\\
%
\end{proof}

In the following, we want to obtain a criterion for tightness in $Z_T.$ Therefore, we introduce the Aldous condition.

\begin{definition}\label{DefinitionAldous}
	 Let $(X_n)_{n\in\N}$ be a sequence of stochastic processes in a Banach space $E.$ Assume that for every $\varepsilon>0$ and $\eta>0$ there is $\delta>0$ such that for every sequence $(\tau_n)_{n\in\N}$ of $[0,T]$-valued stopping times one has
		\begin{align*}
		\sup_{n\in\N} \sup_{0<\theta \le \delta} \Prob \left\{ \Vert X_n((\tau_n+\theta)\land T)-X_n(\tau_n)\Vert_E\ge \eta \right\}\le \varepsilon.
		\end{align*}
		In this case, we say that $(X_n)_{n\in\N}$
		satisfies the Aldous condition $[A].$
\end{definition}

The following Lemma (see \cite{Motyl}, Lemma A.7) gives us a useful consequence of the Aldous condition $[A].$

\begin{lemma} \label{AldousLemma}			
	Let $(X_n)_{n\in\N}$ be a sequence of stochastic processes in a Banach space $E,$ which satisfies the Aldous condition $[A].$ Then, for every $\varepsilon>0$ there exists a measurable subset $A_\varepsilon \subset \D([0,T],E)$ such that
	\begin{align*}
	\Prob^{X_n}(A_\varepsilon)\ge 1-\varepsilon,\qquad
	\lim_{\delta\to 0} \sup_{u\in A_\varepsilon} w_{E}(u,\delta)=0.
	\end{align*}
\end{lemma}

The deterministic compactness result in Proposition $\ref{CompactnessDeterministic}$ and the last Lemma can be used to get the following  tightness criterion  in $Z_T.$

\begin{Prop}\label{TightnessCriterion}
	Let $(X_n)_{n\in\N}$ be a sequence of adapted $\EAdual$-valued processes satisfying the Aldous condition $[A]$ in $\EAdual$ and
	\begin{align*}
	\sup_{n\in\N} \E \left[\sup_{t\in[0,T]}\Vert X_n(t)\Vert_\EA^2\right] <\infty.
	\end{align*}
	Then, the sequence $\left({\Prob}^{X_n}\right)_{n\in\N}$ is tight in $Z_T.$
\end{Prop}

\begin{proof}
	Let $\varepsilon>0.$ With $R_1:= \left(\frac{2}{\varepsilon} \sup_{n\in\N} \E \left[ \sup_{t\in [0,T]}\Vert X_n(t)\Vert_\EA^2\right]\right)^{\frac{1}{2}},$ we obtain
	\begin{align*}
	\Prob\left\{ \sup_{t\in [0,T]}\Vert X_n(t)\Vert_\EA> R_1\right\}\le \frac{1}{R_1^2}\E \left[\sup_{t\in [0,T]}\Vert X_n(t)\Vert_\EA^2\right]\le \frac{\varepsilon}{2}.
	\end{align*}
	By Lemma $\ref{AldousLemma}$, one can use the Aldous condition $[A]$ to find a Borel subset $A$ of $\cadlagHminuseins$ such that
	\begin{align*}
	\inf_{n\in\N} \Prob^{X_n}\left(A\right)\ge 1-\frac{\varepsilon}{2}, \qquad \mbox{ and }\lim_{\delta\to 0} \sup_{u\in A} w_{\EAdual}(u,\delta)=0.
	\end{align*}
	We define $K:= \overline{A\cap B}$ where $B:= \left\{u\in Z_T: \sup_{t\in [0,T]}\Vert X_n(t)\Vert_\EA\le R_1 \right\}.$ This set $K$ is compact in $Z_T$ by Proposition $\ref{CompactnessDeterministic}$ and we can estimate
	\begin{align*}
	\Prob^{X_n}(K)\ge \Prob^{X_n}\left(A\cap B\right)\ge \Prob^{X_n}\left(A\right)-\Prob^{X_n}\left( B^c\right)\ge 1-\frac{\varepsilon}{2}-\frac{\varepsilon}{2}=1-\varepsilon
	\end{align*}
	for all $n\in \N.$
\end{proof}

In metric spaces, one can apply Prokhorov Theorem (see \cite{parthasaraty}, Theorem II.6.7) and Skohorod Theorem (see \cite{Billingsley}, Theorem 6.7.) to obtain a.s.-convergence from tightness. Since $Z_T$ is not a metric space, we use the  following generalization due to Jakubowski \cite{JakubowskiSkorohodTopology} and Brze\'zniak et al \cite{{BrzezniakHausenblasReactionDiffusion}} in the variant of Motyl, \cite{Motyl}, Corollary 7.3.

\begin{Prop}\label{SkohorodJakubowski}
	Let $\mathcal{X}_1 $ be a complete separable metric space and $\mathcal{X}_2$ a topological space such that there is a sequence of continuous functions $f_m: \mathcal{X}_2\to \R$ that separates points of $\mathcal{X}_2.$ Define $\mathcal{X}:=\mathcal{X}_1\times \mathcal{X}_2$ and equip $\mathcal{X}$ with the topology induced by the canonical projections $\pi_j:\mathcal{X}_1\times \mathcal{X}_2\to \mathcal{X}_j.$
	Let $(\Omega, \F, \Prob)$ be a probability space and $\left(\chi_n\right)_{n\in\N}$ be a tight sequence of random variables in $\left(\mathcal{X},\B(\mathcal{X}_1)\otimes \mathcal{A}\right),$ where $\mathcal{A}$ is the $\sigma$-algebra generated by $f_m,$ $m\in\N.$ Assume that there is a random variable $\eta$ in $\mathcal{X}_1$ such that $\Prob^{\pi_1\circ \chi_n}=\Prob^{\eta}.$
	
	Then,  there are a subsequence $\left(\chi_{n_k}\right)_{k\in\N}$ and random variables $\tilde{\chi}_k,$ $\tilde{\chi}$ in $\mathcal{X}$ for $k\in\N$ on a common probability space $(\tilde{\Omega},\tilde{\F},\tildeProb)$ with
	\begin{enumerate}
		\item[i)] $\tildeProb^{\tilde{\chi}_k}=\Prob^{\chi_{n_k}}$ for $k\in\N,$
		\item[ii)] $\tilde{\chi}_k \to \tilde{\chi}$ in $\mathcal{X}$ almost surely for $k\to \infty,$
		\item[iii)] $\pi_1\circ \tilde{\chi}_k=\pi_1\circ \tilde{\chi}$ almost surely.
	\end{enumerate}
\end{Prop}


%
%

\section{Energy Estimates for the  solutions of the Galerkin approximation}

In the following section, we formulate an approximation of \eqref{ProblemMarcus} and prove existence and uniqueness, conservation of the $L^2$-norm as well as uniform bounds of the energy of the solutions to the approximated equation.

Recall from Lemma $\ref{spaceLemma},$ that $S$ has the representation
\begin{align*}
S x=\sum_{m=1}^\infty \lambda_m \skpH{x}{h_m} h_m, \quad x\in \mathcal{D}(S)\mbox{ and } \mathcal{D}(S)=\left\{x\in H: \sum_{m=1}^\infty \lambda_m^2 \vert \skpH{x}{h_m}\vert^2<\infty\right\},
\end{align*}
with an orthonormal basis $\left(h_m\right)_{m\in\N}$ of the complex Hilbert space $\big(H,\skpH{\cdot}{\cdot}\big),$ eigenvalues \\$\lambda_m>0$ such that $\lambda_m\to \infty$ as $m\to \infty.$
For $n\in \N_0,$ we set
\begin{align*}
H_n:=\operatorname{span}\left\{h_m: m\in\N, \lambda_m< 2^{n+1}\right\}
\end{align*}
and denote the orthogonal projection from $H$ to $H_n$ by $P_n,$ i.e.
\begin{align*}
P_n x = \sum_{\lambda_m< 2^{n+1}} \skpH{x}{h_m}h_m, \qquad x\in {H}.
\end{align*}
Since $S$ and $A$ commute by Assumption \ref{spaceAssumptions}, we deduce that  $\norm{P_n}_\LinearOperators{\EA}\le 1$ and
by density of $H$ in $\EAdual,$ we can extend $P_n$ to an operator $P_n: \EAdual\to H_n$ with $\norm{P_n}_{\EAdual\to \EAdual}\le 1$ and
\begin{align}\label{PnInEAdual}
\duality{v}{P_n v}\in \R, \qquad \duality{v}{P_n w}=\skpH{P_n v}{w}, \qquad v\in \EAdual, \quad w\in \EA.
\end{align}
Unfortunately,
 the operators $P_n$, $n\in\N_0,$ are, in general, not uniformly bounded from $\LalphaPlusEins$ to $\LalphaPlusEins.$ Therefore, we have to use another sequence operators introduced in \cite{ExistencePaper} to cut off the noise terms.

\begin{Prop}\label{PaleyLittlewoodLemma}
	There \dela{is}{exists} a sequence $\left(S_n\right)_{n\in\N_0}$ of self-adjoint operators $S_n: H \to H_n$ for $n\in\N_0$ with
	$S_n \psi \to \psi$ in $E_A$ for $n\to \infty$ and $\psi \in E_A$  and the uniform norm estimates
	\begin{align}\label{SnUniformlyBounded}
	\sup_{n\in\N_0}\norm{S_n}_{{\mathcal{L}(H)}}\le 1, \quad \sup_{n\in\N_0} \norm{S_n}_{\mathcal{L}(\EA)}\le 1, \quad \sup_{n\in\N_0} \norm{S_n}_{\mathcal{L}(L^{\alpha+1})}<\infty.
	\end{align}
\end{Prop}

A proof of this result can be in \cite{ExistencePaper}, Proposition 5.2. For convenience of the reader, we present an alternative proof. 

\begin{proof}
\emph{Step 1.} 
We take a function $\rho\in C_c^\infty(0,\infty)$ with $\operatorname{supp} \rho \subset [\frac{1}{2},2]$ and $\sum_{m\in\Z} \rho(2^{-m} t)=1$ for all $t>0.$ For the existence of $\rho$ with these properties, we refer to \cite{bergh1976interpolation}, Lemma 6.1.7. 
Then, we fix $n\in\N_0$ and define 
\begin{align*}
s_n: (0,\infty)\to \C,\qquad s_n(\lambda):=\sum_{m=-\infty}^{n}\rho(2^{-m}\lambda).
\end{align*}
Let $k\in \Z$ and $\lambda\in [2^{k-1},2^{k}).$  From $\operatorname{supp} \rho \subset [\frac{1}{2},2],$ we infer
\begin{align*}
1=\sum_{m=-\infty}^\infty \rho(2^{-m}\lambda)=\rho(2^{-(k-1)}\lambda)+\rho(2^{-k}\lambda)=\sum_{m=-\infty}^k \rho(2^{-m}\lambda).
\end{align*}
In particular
\begin{align}\label{SnMultiplier}
s_n(\lambda)=\begin{cases}
1,\hspace{2cm}   &\lambda\in (0,2^n),\\
\rho(2^{-n}\lambda), &\lambda\in [2^{n},2^{n+1}),\\
0,&\lambda\ge 2^{n+1}.\\
\end{cases}
\end{align}
We define $S_n:=s_n(S)$ for $n\in\N_0.$ Since $s_n$ is real-valued and bounded by $1,$ the operator $S_n$ is selfadjoint with $\norm{S_n}_\LinearOperators{H}\le 1.$ Furthermore, $S_n$ and $A$ commute due to the assumption that $S$ and $A$ commute. In particular, this implies $\norm{S_n}_{\mathcal{L}(\EA)}\le 1$ and $S_n \psi \to \psi$ for all $\psi\in \EA$ by the convergence property of the Borel functional calculus. Moreover, the range of $S_n$ is contained in $H_n$ since we have the representation
\begin{align*}
S_n x=\sum_{\lambda_m < 2^n} \skpH{x}{h_m} h_m+ \sum_{\lambda_m \in [ 2^n,2^{n+1})} \rho(2^{-n}\lambda_m) \skpH{x}{h_m} h_m,\quad x\in H,
\end{align*}
as a consequence of \eqref{SnMultiplier}. 

\emph{Step 2.} Next, we show the uniform estimate in $\LalphaPlusEins$ based on a spectral multiplier theorem by Kunstmann and Uhl, \cite{kunstmannUhl}, for operators with generalized Gaussian bounds.
In view of Theorem 5.3 in \cite{kunstmannUhl}, Lemma 2.19 and Fact 2.20 in \cite{Uhl}, it is sufficient to show that $s_n$ satisfies the Mihlin condition
\begin{align}\label{MihlinFn}
\sup_{\lambda>0} \vert \lambda^k s_n^{(k)}(\lambda)\vert\le C_k,\qquad k=0,\dots, \gamma,
\end{align}
for some $\gamma\in \N$ uniformly in $n\in\N_0.$  This  can be verified by the calculation
\begin{align*}
\sup_{\lambda>0} \vert \lambda^k s_n^{(k)}(\lambda)\vert=&\sup_{\lambda\in [2^{n},2^{n+1})} \vert \lambda^k s_n^{(k)}(\lambda)\vert=\sup_{\lambda\in [2^{n},2^{n+1})} \left\vert \lambda^k \frac{\df^k }{\df \lambda^k}\rho(2^{-n}\lambda)\right\vert\le 2^k \norm{\rho^{(k)}}_\infty
\end{align*}
for all $k\in \N_0.$ 
\end{proof}	

We set
\begin{align*}
\Bn(l)=\sumM l_m S_n B_m S_n,\qquad n\in \N, \quad l\in \R^N
\end{align*}
and 
	\begin{align}\label{eqn-approx initial data}
	\widetilde{u_{0,n}}:=
	\begin{cases}
S_n u_0 \frac{\norm{u_0 }_H}{\norm{S_n u_0 }_H}, &  S_n u_0\neq 0,\\
0,& S_n u_0=0.		
	\end{cases}
	\end{align}
From $S_n u_0 \to u_0$ in $H,$  we infer 
\begin{align}\label{initialValueConvergence}
	\widetilde{u_{0,n}}\to u_0,\qquad n\to\infty.
\end{align}
Moreover, there is $C_0>0$ such that we have
\begin{align}\label{initialValueBoundedness}
	1\le\frac{\norm{u_0 }_H}{\norm{S_n u_0 }_H}\le C_0.
\end{align} 
for $n\ge n_0(u_0):=\min\{n\in\N: S_n u_0\neq 0\}\in\N\cup \left\{\infty \right\}.$ 
For $n\in \N,$ we consider the Galerkin equation
\begin{align}\label{marcusGalerkin}
u_n(t)  & = \widetilde{u_{0,n}} -\im \int_0^t\left( A u_n(s)+P_n F(u_n(s))\right) \,\df s + \int_{0}^{t}\! \int_{\left\{\vert l \vert\le 1\right\}} \!\left[\groupBn u_n(s-) - u_n(s-)\right]\, \tilde{\eta}(\df s,\df l)\nonumber\\
&\quad + \int_{0}^{t}\! \int_{\left\{\vert l \vert\le 1\right\}} \! \left\{ \groupBn u_n(s) - u_n(s) + \im \Bn(l) u_n(s)\right\}\, \nu(\df l)\df s, \;\; t \in [0,T].
\end{align}
In order to prove the global wellposedness of \eqref{marcusGalerkin} and estimates for the solution $u_n$ uniformly in $n\in\N,$ we need some auxiliary Lemmata. We start with properties of the operators $\Bn(l).$

\begin{lemma}\label{BnProperties}
	Let $n\in\N$ and $l\in \R^N.$ Then, we have
	\begin{align*}
	\norm{\Bn(l)}_{\LinearOperators{\Lzwei}}\le \vert l\vert b_\Lzwei^\frac{1}{2},\qquad \norm{\Bn(l)}_{\LinearOperators{\EA}}\le \vert l\vert b_\EA^\frac{1}{2},\qquad \norm{\Bn(l)}_{\LinearOperators{L^{\alpha+1}}}\le \vert l\vert b_{{\alpha+1}}^\frac{1}{2} \sup_{n\in\N} \norm{S_n}_\LinearOperators{L^{\alpha+1}}^2.
	\end{align*}
	Moreover, $\left(\groupBnT\right)_{t\in\R}$ is a group of unitary operators on $\Lzwei$ with
	\begin{align*}
	\norm{ \groupBnT}_\LinearOperators{\EA} \le e^{\vert t\vert \vert l\vert b_\EA^\frac{1}{2} },\qquad \norm{ \groupBnT}_\LinearOperators{L^{\alpha+1}}\le e^{\vert t\vert \vert l\vert b_{\alpha+1}^\frac{1}{2} \sup_{n\in\N} \norm{S_n}_\LinearOperators{L^{\alpha+1}}^2},\qquad t\in\R.
	\end{align*}	
\end{lemma}

\begin{proof}
	By the boundedness of $\left(S_n\right)_{n\in\N}\in \LinearOperators{L^{\alpha+1}}^\N,$ we deduce that
	\begin{align}\label{BnLpEstimate}
	\norm{\Bn(l)}_{\LinearOperators{L^{\alpha+1}}}&\le \sumM \vert l_m\vert \norm{S_n B_m S_n }_\LinearOperators{L^{\alpha+1}}\le \vert l\vert \left(\sumM \norm{B_m}_\LinearOperators{L^{\alpha+1}}^2\right)^\frac{1}{2} \sup_{n\in\N} \norm{S_n}_\LinearOperators{L^{\alpha+1}}^2\nonumber\\
	&= \vert l\vert b_{{\alpha+1}}^\frac{1}{2} \sup_{n\in\N} \norm{S_n}_\LinearOperators{L^{\alpha+1}}^2.
	\end{align}	
	The estimates of $\Bn(l)$ in spaces  $\Lzwei$ and $\EA$ can be shown analogously using $\norm{S_n}_\LinearOperators{\Lzwei}=1$ and $\norm{S_n}_\LinearOperators{\EA}=1.$ Since $S_n$ and $B_m$ are self-adjoint  on $\Lzwei$ for $n\in \N$ and $m\in \left\{1,\dots, M\right\},$ the Stone Theorem yields that $\left(\groupBnT\right)_{t\in\R}$ is a unitary group  on $\Lzwei.$ Moreover,
	\begin{align*}
	\norm{ \groupBnT x}_\EA \le e^{\vert t\vert \norm{ \Bn(l) }_{\mathcal{L}(\EA)} }\norm{x}_\EA
	\le e^{\vert t\vert\vert l\vert b_\EA^\frac{1}{2} }\norm{x}_\EA,\qquad x\in \EA,\quad t\in \R,
	\end{align*}
	\begin{align*}
	\norm{\groupBnT x}_{L^{\alpha+1}}
	&\le e^{\vert t\vert\norm{ \Bn(l) }_{\mathcal{L}(L^{\alpha+1})} }\norm{x}_{L^{\alpha+1}}\nonumber\\&\le e^{\vert t\vert\vert l\vert b_{\alpha+1}^\frac{1}{2} \sup_{n\in\N} \norm{S_n}_\LinearOperators{L^{\alpha+1}}^2}\norm{x}_{L^{\alpha+1}},\qquad x\in \LalphaPlusEins,\quad t\in \R.
	\end{align*}	
\end{proof}

In the next Lemma inspired by Lemma 2.2 in \cite{Brz+Manna_2017_SLLGELevy} \dela{\coma{TODO: Reference to Brzezniak Manna}}, we show how to control the differences in \eqref{marcusGalerkin} in the $H$-norm.

\begin{lemma}\label{AldousDifferences} For every $n\in\N,$ $l\in B$ and $x\in \Lzwei,$ the following inequalities hold:
	\begin{align*}
	\norm{\groupBn x-x}_\Lzwei\le b_\Lzwei^\frac{1}{2}\vert l\vert \norm{x}_\Lzwei,
	\end{align*}
	\begin{align*}
	\norm{\groupBn x-x+\im \Bn(l) x}_\Lzwei\le \frac{1}{2} b_\Lzwei\vert l\vert^2 \norm{x}_\Lzwei.
	\end{align*}
\end{lemma}

\begin{proof}
	The identities
	\begin{align*}
	\groupBn x-x= \int_0^1 \frac{\df}{\df t} \groupBnT x \df t= -\im \Bn(l)\int_0^1 \groupBnT x \df t
	\end{align*}
	and
	\begin{align*}
	\groupBn x-x+\im \Bn(l)x=\int_0^1 \int_0^s \frac{\df^2}{\df t^2} \groupBnT x \df t \df s=-\Bn(l)^2 \int_0^1 \int_0^s  \groupBnT x \df t \df s
	\end{align*}
	and Lemma \ref{BnProperties} lead to
	\begin{align*}
	\norm{\groupBn x-x}_\Lzwei\le \norm{\Bn(l)}_\LinearOperators{\Lzwei} \int_0^1 \norm{\groupBnT x}_\Lzwei \df t\le b_\Lzwei^\frac{1}{2} \vert l\vert \norm{x}_\Lzwei,
	\end{align*}
	\begin{align*}
	\norm{\groupBn x-x+\im \Bn(l)x}_\Lzwei\le \norm{\Bn(l)}_\LinearOperators{\Lzwei}^2 \int_0^1 \int_0^s  \norm{\groupBnT x}_\Lzwei \df t \df s\le \frac{1}{2}b_\Lzwei\vert l\vert^2 \norm{x}_\Lzwei.
	\end{align*}
\end{proof}


Next, we prove the well-posedness of the Galerkin equation. Moreover, we show that the Marcus noise and the approximation do not destroy the mass conservation which is  well-known for the deterministic nonlinear Schr\"odinger equation.

\begin{Prop}\label{massConservationGalerkin}
	For each $n\in\N,$ there is a unique global strong solution $u_n\in \D([0,T],H_n)$ of \eqref{marcusGalerkin} and we have the equality 
	\begin{align}\label{massConservationEstimate}
		\norm{u_n(t)}_H=\norm{\widetilde{u_{0,n}}}_H = \norm{u_0}_H
	\end{align}
	almost surely for all $t\in [0,T].$
\end{Prop}

\begin{proof}
\emph{Step 1.} We fix $n\in\N.$ To obtain a global solution, we regard $H_n$ as a finite dimensional real Hilbert space equipped with the scalar product
$\skpHn{u}{v}:=\Real \skpH{u}{v}$ and check the assumptions of  \cite{AlbeverioBrzezniakPoissonODE}, Theorem 3.1 for the coefficients defined by
\begin{align*}
	\xi&=\widetilde{u_{0,n}},\qquad \sigma(u)=0, \\
	 b(u)&=-\im A u-\im P_n F(u)+ \int_{\left\{\vert l \vert\le 1\right\}} \! \left\{ \groupBn u - u + \im \Bn(l) u\right\}\, \nu(\df l),\\
	g(u,l)&=\!\left[\groupBn u - u\right]
\end{align*}	
for $u\in H_n$ and $l\in B.$
Let $R>0$. We take $u,v\in H_n$ such that $\norm{u}_H,\norm{v}_H\le R$ and estimate
\begin{align}\label{GalerkinLipschitzStart}
	\norm{b(u)-b(v)}_H\le& \norm{A|_{H_n}}_\LinearOperators{H} \norm{u-v}_H+\norm{F(u)-F(v)}_H
	\nonumber\\&+\int_{\left\{\vert l \vert\le 1\right\}} \! \norm{ \groupBn (u-v) - (u-v) + \im \Bn(l) (u-v)}_H\, \nu(\df l).
\end{align}
By Lemma \ref{AldousDifferences} and \eqref{LevyMeasure}
\begin{align}\label{GalerkinNoise}
	\int_{\left\{\vert l \vert\le 1\right\}} \! \norm{ \groupBn (u-v) - (u-v) + \im \Bn(l) (u-v)}_H\, \nu(\df l)&\le \frac{1}{2} b_\Lzwei\int_{\left\{\vert l \vert\le 1\right\}} \vert l\vert^2 \nu(\df l)\norm{u-v}_\Lzwei\nonumber\\
	&\lesssim \norm{u-v}_\Lzwei.
\end{align}
To estimate the nonlinearity, we use the equivalence of all norms in $H_n$ and \eqref{nonlinearityLocallyLipschitz} to get 
\begin{align}\label{GalerkinNonlinearity}
	\norm{P_n F(u)-P_n F(v)}_H&\lesssim_n \norm{P_n F(u)-P_n F(v)}_{\EAdual}\lesssim \norm{ F(u)- F(v)}_{L^{\frac{\alpha+1}{\alpha}}}\nonumber\\
	&
	\lesssim \left(\norm{u}_\LalphaPlusEins+\norm{v}_\LalphaPlusEins\right)^{\alpha-1} \norm{u-v}_\LalphaPlusEins\nonumber\\
	&\lesssim \left(\norm{u}_H+\norm{v}_H\right)^{\alpha-1} \norm{u-v}_H\lesssim_R \norm{u-v}_H.
\end{align}
We insert \eqref{GalerkinNonlinearity} and 	\eqref{GalerkinNoise} in \eqref{GalerkinLipschitzStart} to get a constant $C=C(R)$ such that
\begin{align}\label{Lipschitz1}
	\norm{b(u)-b(v)}_H\le C \norm{u-v}_H.
\end{align}
Moreover, we have
\begin{align}\label{Lipschitz2}
	\int_{\left\{\vert l \vert\le 1\right\}} \norm{g(u,l)-g(v,l)}_H^2 \nu(\df l)\le b_\Lzwei\int_{\left\{\vert l \vert\le 1\right\}} \vert l\vert^2 \nu(\df l) \norm{u-v}_\Lzwei^2\lesssim \norm{u-v}_\Lzwei^2
\end{align}
where we used Lemma \ref{AldousDifferences} and \eqref{LevyMeasure}. To check the one-sided linear growth condition, we use \eqref{nonlinearityComplex} and \eqref{GalerkinNoise} for $v=0$ and obtain a constant $K_1>0$ with
\begin{align}\label{LinearGrowth}
	2\skpHn{u}{b(u)}+\int_{\left\{\vert l \vert\le 1\right\}} \norm{g(u,l)}_H^2 \nu(\df l)\le& 2 \norm{A|_{H_n}}_\LinearOperators{H} \norm{u}_H^2+2 \Real \skpH{u}{-\im F(u)}\nonumber\\&+2 \norm{u}_H \int_{\left\{\vert l \vert\le 1\right\}} \! \norm{ \groupBn u - u + \im \Bn(l) u}_H\, \nu(\df l)\nonumber\\
	\le& K_1 \norm{u}_H^2.
\end{align}
In view of \eqref{Lipschitz1}, \eqref{Lipschitz2} and \eqref{LinearGrowth}, we can apply Theorem 3.1 of \cite{AlbeverioBrzezniakPoissonODE} and get a unique global strong solution of \eqref{marcusGalerkin} for each $n\in \N.$\\

\emph{Step 2.} It remains to show \eqref{massConservationEstimate}.	
The function $\mathcal{M}: H_n \to \R$  defined by  $\mathcal{M}(v):=\norm{v}_\Lzwei^2$ for $v\in H_n$ is   continuously Fr\'{e}chet-differentiable with
\begin{align*}
\mathcal{M}^\prime[v]h_1&= 2 \Real \skpLzwei{v}{ h_1}, \qquad
\end{align*}
for $v, h_1, h_2\in H_n.$ By the It\^o  formula and \eqref{marcus2}, we get almost surely,  for all $t\in [0,T]$,
\begin{align*}
\norm{u_n(t)}_\Lzwei^2=&\norm{\widetilde{u_{0,n}}}_\Lzwei^2+2 \int_0^t \Real \skpLzwei{ u_n(s)}{ -\im A u_n(s)-\im P_n F\left(u_n(s)\right)}\df s
\\&
+\int_0^t \int_{\left\{\vert l \vert\le 1\right\}} \left[\norm{\groupBn u_n(s-)}_\Lzwei^2-\norm{u_n(s-)}_\Lzwei^2\right]\tilde{\eta}(\df l,\df s)\\&
+\int_0^t \int_{\left\{\vert l \vert\le 1\right\}} \left[\norm{\groupBn u_n(s)}_\Lzwei^2-\norm{u_n(s)}_\Lzwei^2  \right]\nu(\df l)\df s\\
&-2 \int_0^t \int_{\left\{\vert l \vert\le 1\right\}} \Real \skpLzwei{u_n(s)}{-\im \sumM l_m S_n B_m S_n u_n(s)} \nu(\df l)\df s.
\end{align*}
 By
\begin{align*}
\Real \skpLzwei{ v}{ -\im A v}&=\Real \left[\im \Vert \sqrtA v\Vert_\Lzwei^2\right]=0,\qquad
\Real \skpLzwei{ v}{ -\im P_n F\left( v\right)}=0, \qquad
\Real \skpLzwei{v}{\im B_m v}=0
\end{align*}
for $v\in H_n$ and the fact that $S_n \B(l)S_n $ is self-adjoint and hence, $\groupBn$ unitary, this simplifies to
\begin{align*}
\norm{u_n(t)}_\Lzwei^2=&\norm{\widetilde{u_{0,n}}}_\Lzwei^2=\norm{u_0}_\Lzwei^2
\end{align*}
almost surely for all $t\in [0,T].$
\end{proof}

Recall that by Assumption $\ref{nonlinearAssumptions},$ the nonlinearity $F$ has a real antiderivative denoted by $\Fhat.$
The second ingredient for uniform estimates in $\EA$ is to control the energy associated to the NLS.

\begin{definition}
	We define the energy $\energy$ \dela{of $u\in \EA$} function  by
	\begin{align*}
	\energy(u):= \frac{1}{2} \Vert A^{\frac{1}{2}} u \Vert_{H}^2+\Fhat(u),\qquad u\in \EA.
	\end{align*}	
\end{definition}
Note that $\energy(u)$ is well  defined for every $u\in \EA$ by the continuity of the  embedding $\EA \hookrightarrow L^{\alpha+1}(M).$ The compactness of this embedding formulated in Assumption \ref{spaceAssumptions} is not needed here.
Before we estimate the energy of the solutions $u_n$ of \eqref{marcusGalerkin}, we need some preparations.

\begin{lemma}\label{lem-EstimateItoDifferencesEnergy}
	\begin{trivlist}
		\item[\;a)] There is a constant $C=C(b_\EA,b_{\alpha+1},\alpha,F)>0$ such that for every $n\in\N,$ we have
		\begin{align*}
		\vert \energy(\groupBn x)-\energy(x)\vert \le &C \vert l\vert  \left(\norm{x}_\EA^2+\norm{x}_{L^{\alpha+1}}^{\alpha+1}\right)
		\end{align*}
		for all $x\in H_n,$ and $l\in \R^N$ with $\vert l\vert\le 1.$
		\item[\;b)]There is a constant $C=C(b_\EA,b_{\alpha+1},q, \alpha,F)>0$ such that for every $n\in\N,$ we have
		\begin{align*}
		\vert \energy(\groupBn x)-\energy(x)+\energy'[x](\im \Bn(l) x)\vert \le &C \vert l\vert^2  \left(\norm{x}_\EA^2+\norm{x}_{L^{\alpha+1}}^{\alpha+1}\right)
		\end{align*}
		for all $x\in H_n,$ and $l\in \R^N$ with $\vert l\vert\le 1.$
	\end{trivlist}
\end{lemma}

\begin{proof}
	\emph{ad a):} The map
	$\energy$ is twice continuously Fr\'{e}chet-differentiable with
	\begin{align*}
	\energy'[v]h=& \Real \duality{Av+F(v)}{ h},\\
	\energy''[v](h_1,h_2)=&\Real \skpLzwei{\sqrtA h_1}{\sqrtA h_2} +\Real \duality{F'[v]h_1}{h_2}
	\end{align*}
	for $v,h_1,h_2\in H_n.$ Let us fix $x\in H_n$ and  $l\in B$. Then, we get
	\begin{align}\label{ItoDifferencesStart}
	\energy(\groupBn x)-\energy(x)=&\int_0^1 \frac{\df }{\df t} \energy(\groupBnT x)\df t=\int_0^1 \energy'[\groupBnT x]\left(-\im \Bn(l) \groupBnT x\right)\df t\nonumber \\=&\int_0^1 \Real \dualityBig{A \groupBnT x+F(\groupBnT x)}{ -\im \Bn(l) \groupBnT x}			 \df t.
	\end{align}
	We define $f: [0,1]\times \R^N\to [0,\infty)$ by
	\begin{align*}
	f(t,l):=\max\left\{1,e^{2 t \vert l\vert b_\EA^\frac{1}{2} }+e^{(\alpha+1)t\vert l\vert b_{\alpha+1}^\frac{1}{2} \sup_{n\in\N} \norm{S_n}_\LinearOperators{L^{\alpha+1}}^2}\right\},\qquad t\in [0,1],\quad l\in \R^N,
	\end{align*}
	and by the properties of $\Bn(l)$ from Lemma \ref{BnProperties}, we estimate  the integrand of 	\eqref{ItoDifferencesStart}:		
	\begin{align}\label{ItoDifferencesAuxiliaryOne}
	\vert\skpLzwei{A \groupBnT x}{-\im   \Bn(l) \groupBnT x}\vert
	&\le \Vert \sqrtA \groupBnT x\Vert_{L^2} \Vert \sqrtA \Bn(l) \groupBnT x \Vert_{L^2}\nonumber\\
	&\le e^{t \vert l\vert b_\EA^\frac{1}{2} }\norm{x}_\EA \vert l\vert b_\EA^\frac{1}{2} \Vert   \groupBnT x\Vert_\EA\nonumber\\
	&\le e^{2 t \vert l\vert b_\EA^\frac{1}{2} }\vert l\vert b_\EA^\frac{1}{2}\norm{x}_\EA^2
	\end{align}
	and
	\begin{align}\label{ItoDifferencesAuxiliaryTwo}
	\left\vert  \dualityBig{F(\groupBnT x)}{ -\im \Bn(l) \groupBnT x}\right\vert
	&\le \Vert F(\groupBnT x)\Vert_{L^\frac{\alpha+1}{\alpha}} \Vert   \Bn(l) \groupBnT x \Vert_{L^{\alpha+1}}\nonumber\\
	&\le C_{F,1} \Vert   \Bn(l)\Vert_{\LinearOperators{L^{\alpha+1}}}  \Vert \groupBnT x\Vert_{L^{\alpha+1}}^{\alpha+1} \nonumber\\
	&\le C_{F,1} \vert l\vert b_{{\alpha+1}}^\frac{1}{2} \sup_{n\in\N} \norm{S_n}_\LinearOperators{L^{\alpha+1}}^2 \norm{x}_{L^{\alpha+1}}^{\alpha+1}\nonumber\\&\hspace{2cm}e^{(\alpha+1)t\vert l\vert b_{\alpha+1}^\frac{1}{2} \sup_{n\in\N} \norm{S_n}_\LinearOperators{L^{\alpha+1}}^2}.
	\end{align}
	
	We obtain
	\begin{align*}
	&\vert \energy(\groupBn x)-\energy(x)\vert \\&\le \vert l\vert \max\left\{b_\EA^\frac{1}{2},C_{F,1} b_{\alpha+1}^\frac{1}{2} \sup_{n\in\N}\norm{S_n}_\LinearOperators{L^{\alpha+1}}^2\right\} \left(\norm{x}_\EA^2+\norm{x}_{L^{\alpha+1}}^{\alpha+1}\right) \int_0^1 f(t,l)  \df t
	\end{align*}
	and the assertion follows from
	\begin{align}\label{EstimateIntegraloverF}
	\int_0^1 f(t,l)  \df t&= \int_0^1\max\left\{1,e^{2  t \vert l\vert b_\EA^\frac{1}{2} }+e^{(\alpha+1)t\vert l\vert b_{\alpha+1}^\frac{1}{2} \sup_{n\in\N} \norm{S_n}_\LinearOperators{L^{\alpha+1}}^2}\right\} \df t\nonumber\\
	&\le \max\left\{1,e^{2   b_\EA^\frac{1}{2} }+e^{(\alpha+1) b_{\alpha+1}^\frac{1}{2} \sup_{n\in\N} \norm{S_n}_\LinearOperators{L^{\alpha+1}}^2}\right\}<\infty,\qquad \vert l\vert \le 1.
	\end{align}

	\emph{ad b):} Let us fix $x\in H_n$ and  $l\in B$. We start with the identity
	\begin{align*}
	&\energy(\groupBn x)-\energy(x)+\energy'[x](\im \Bn(l)) x=\int_0^1 \left(\frac{\df }{\df s} \energy(\groupBnS x)-\frac{\df }{\df s} \energy(\groupBnS x)\bigg\vert _{s=0}\right)\df s\\
	&\hspace{3cm}=\int_0^1 \int_0^s\frac{\df^2 }{\df t^2} \energy(\groupBnT x)\df t \df s\\
	&\hspace{3cm}=\int_0^1  \int_0^s\energy'[\groupBnT x]\left(-\Bn(l)^2 \groupBnT x\right)\df t \df s\\
	&\hspace{3,3cm}+ \int_0^1  \int_0^s\energy''[\groupBnT x]\left(-\im\Bn(l) \groupBnT x,-\im\Bn(l) \groupBnT x\right)\df t \df s\\
	&\hspace{3cm}=:I_1+I_2.
	\end{align*}
	As above
	\begin{align*}
	\vert I_1\vert \le& \vert l\vert^2 \max\left\{ b_\EA,C_{F,1} b_{\alpha+1} \sup_{n\in\N}\norm{S_n}_\LinearOperators{L^{\alpha+1}}^4\right\} \left(\norm{x}_\EA^2+\norm{x}_{L^{\alpha+1}}^{\alpha+1}\right) \int_0^1 f(t,l)  \df t.
	\end{align*}
	We further decompose
	$I_2=I_{2,1}+I_{2,2}$
	with
	\begin{align*}
	I_{2,1}=   \int_0^1  \int_0^s   \norm{ \sqrtA \Bn(l) \groupBnT x}_{L^2}^2  \df t \df s,
	\end{align*}
	\begin{align*}
	I_{2,2}=   \int_0^1  \int_0^s    \Real\dualityBig{F'[\groupBnT x]\Bn(l) \groupBnT x}{\Bn(l) \groupBnT x} \df t \df s.
	\end{align*}
	By  Lemma \ref{BnProperties},
	\begin{align*}
	\vert I_{2,1}\vert \le&    \int_0^1  \int_0^s   \vert l\vert^2 b_\EA \norm{\groupBnT x}_\EA^2 \df t \df s
	\le  \norm{x}_\EA^2\vert l\vert^2 b_\EA   \int_0^1     f(t,l)\df t.
	\end{align*}
	Moreover, the estimate
	\begin{align*}
	&\left\vert\dualityBig{F'[\groupBnT x]\Bn(l) \groupBnT x}{\Bn(l) \groupBnT x}\right\vert\\
	&\hspace{3cm}\le \norm{F'[\groupBnT x]\Bn(l) \groupBnT x}_{L^\frac{\alpha+1}{\alpha}} \norm{\Bn(l) \groupBnT x}_{L^{\alpha+1}}\nonumber\\
	&\hspace{3cm}\le C_{F,2} \norm{\Bn(l)}_\LinearOperators{L^{\alpha+1}}^2\norm{\groupBnT x}_{L^{\alpha+1}}^{\alpha+1}\nonumber\\
	&\hspace{3cm}\le C_{F,2}\vert l\vert^2 b_{\alpha+1} \sup_{n\in\N}\norm{S_n}_\LinearOperators{L^{\alpha+1}}^4 f(t,l) \norm{x}_{L^{\alpha+1}}^{\alpha+1}
	\end{align*}
	yields
	\begin{align*}
	\vert I_{2,2}\vert \le&   \int_0^1  \int_0^s \left\vert\dualitybig{F'[\groupBnT x]\Bn(l) \groupBnT x}{\Bn(l) \groupBnT x}\right\vert \df t \df s\\
	\le& C_{F,2}\vert l\vert^2 b_{\alpha+1} \sup_{n\in\N}\norm{S_n}_\LinearOperators{L^{\alpha+1}}^4  \norm{x}_{L^{\alpha+1}}^{\alpha+1} \int_0^1 f(t,l)\df t
	\end{align*}
	and finally, we find a constant $C=C(b_{\alpha+1},b_{\EA},\sup_{n\in\N}\norm{S_n}_\LinearOperators{L^{\alpha+1}},F)$ such that
	\begin{align*}
	\left\vert\energy(\groupBn x)-\energy(x)+\energy'[x](\im \Bn(l) x)\right\vert\le C\vert l\vert^2\left(\norm{x}_\EA^2+\norm{x}_{L^{\alpha+1}}^{\alpha+1}\right)   \int_0^1   f(t,l)\df t
	\end{align*}
	and the second assertion also follows from \eqref{EstimateIntegraloverF}.
\end{proof}

The next observation will be useful to simplify the following arguments based on the  Gronwall Lemma to estimate of the energy. It has already appeared in \cite{ExistencePaper}, Lemma 5.6, but we need it in a slightly more general form.

\begin{lemma}\label{LemmaYOmegaNachLzweiZeit}
	Let $r\in [1,\infty),$ $q\in (1,\infty),$ $\varepsilon>0,$ $T>0$ and $X\in L^r(\Omega,L^\infty(0,\dela{t}{T})).$ Then,
	\begin{align*}
	\norm{X}_{L^r(\Omega,L^q(0,t))}\le \varepsilon \norm{X}_{L^r(\Omega,L^\infty(0,t))} +\varepsilon^{1-q} \frac{1}{q} \left(1-\frac{1}{q}\right)^{q-1}\int_0^t \norm{X}_{L^r(\Omega,L^\infty(0,s))} \df s,\qquad t\in[0,T].
	\end{align*}
\end{lemma}

\begin{proof}
	As a consequence of Young's inequality, we obtain
	\begin{align}\label{YoungEpsilon}
	a^{1-\frac{1}{q}} b^\frac{1}{q}\le \varepsilon a+\varepsilon^{1-q} \frac{1}{q} \left(1-\frac{1}{q}\right)^{q-1}b,\qquad a,b\ge 0,\quad\varepsilon>0.
	\end{align}
	Then, interpolation of $L^q(0,t)$ between $L^\infty(0,t)$ and $L^1(0,t)$ and \eqref{YoungEpsilon} yield
	\begin{align*}
	\norm{X}_{L^q(0,t)}\le \norm{X}_{L^\infty(0,t)}^{1-\frac{1}{q}} \norm{X}_{L^1(0,t)}^\frac{1}{q}\le \varepsilon \norm{X}_{L^\infty(0,t)} +\varepsilon^{1-q} \frac{1}{q} \left(1-\frac{1}{q}\right)^{q-1} \norm{X}_{L^1(0,t)}.
	\end{align*}
	Now, we take the $L^r(\Omega)$-norm and apply Minkowski's inequality to get
	\begin{align*}
	\norm{X}_{L^r(\Omega,L^q(0,t))}&\le \varepsilon \norm{X}_{L^r(\Omega,L^\infty(0,t))} +\varepsilon^{1-q} \frac{1}{q} \left(1-\frac{1}{q}\right)^{q-1}\int_0^t \norm{X(s)}_{L^r(\Omega)} \df s\\
	&\le\varepsilon \norm{X}_{L^r(\Omega,L^\infty(0,t))} +\varepsilon^{1-q} \frac{1}{q} \left(1-\frac{1}{q}\right)^{q-1}\int_0^t \norm{X}_{L^r(\Omega,L^\infty(0,s))} \df s.
	\end{align*}
\end{proof}	
Now, we are ready prove that the solutions of \eqref{marcusGalerkin} have uniform energy estimates and satisfy the Aldous condition.

	\begin{Prop}\label{EstimatesGalerkinSolution}
		Let us assume Assumption $\ref{focusing}$ i). Then, the following assertions hold: 
		\begin{trivlist}
			\item[\;a)]  For all $q\in [1,\infty)$ there exists $C=C(\energy{(u_0)},T,b_\EA,b_{\alpha+1},q, \alpha,F)>0$  such that
			\begin{align*}
			\sup_{n\in\N}\E \Big[\sup_{t\in[0,T]} \left[\norm{u_n(t)}_\Lzwei^2+\energy(u_n(t))\right]^q\Big]\le  C.
			\end{align*}
%
			\item[\;b)] The sequence $(u_n)_{n\in\N}$ satisfies the Aldous condition $[A]$ in $\EAdual.$
		\end{trivlist}
		\item[\;c)] 	The sequence $\left({\Prob}^{u_n}\right)_{n\in\N}$ is tight in $Z_T$.
	\end{Prop}
	
	\begin{proof}
\emph{Ad c):} Follows from the two other parts by applying Proposition \ref{TightnessCriterion}.\\
		\emph{Ad a):} Since $\widetilde{u_{0,n}}=0$ already implies $u_n\equiv 0,$ we may assume $\widetilde{u_{0,n}}\neq 0$ without loss of generality.  Furthermore, we only prove the assertion for $q> 2.$ The case $q\in [1,2]$ is a simple consequence of the H\"older inequality.
		Recall that the energy $\energy$ is twice Frechet differentiable.
		In particular, the function $\energy^{\prime}$ is H\"older continuous.  Hence, we can use Proposition \ref{massConservationGalerkin} and the It\^o  formula \ref{T2} to deduce		
					\begin{align}\label{ItoEnergyStart}
					\frac{1}{2}\norm{u_n(s)}_\Lzwei^2+&\energy\left(u_n(s)\right)=\frac{1}{2}\norm{\widetilde{u_{0,n}}}_\Lzwei^2+\energy\left(\widetilde{u_{0,n}}\right)\nonumber\\
					&+ \int_0^s \Real \duality{A u_n(r)+F(u_n(r))}{ -\im A u_n(r)-\im P_n F(u_n(r))}\df r\nonumber\\
					&+\int_0^s \int_{\left\{\vert l \vert\le 1\right\}} \left[\energy(\groupBn u_n(r-))-\energy(u_n(r-))\right] \tilde{\eta}(\df l,\df r)\nonumber\\
					&+\int_0^s \int_{\left\{\vert l \vert\le 1\right\}} \left[\energy(\groupBn u_n(r))-\energy(u_n(r))+\energy^{\prime}[u_n(r)]\left(\im \Bn(l) u_n(r)\right)\right]  \nu(\df l)\df r\nonumber\\
					&=:\frac{1}{2}\norm{\widetilde{u_{0,n}}}_\Lzwei^2+\energy\left(\widetilde{u_{0,n}}\right)+I_1(s)+I_2(s)+I_3(s)
					\end{align}
					almost surely for all $s\in [0,T].$
		The first integral $I_1(s)$ cancels due to the following three identities which hold for all 		for all $v\in H_n$: 				
		\begin{align*}
		\Real \duality{F(v)}{ -\im P_n F(v)}=\Real \left[\im \duality{F(v)}{  P_n F(v)}\right]=0;
		\end{align*}
		\begin{align*}
		\Real \left[\duality{A v}{ -\im P_n F(v)}+\duality{F(v)}{ -\im A v}\right]
		&=\Real \left[-\duality{A v}{ \im F(v)}+\overline{\duality{  A v}{\im F(v)}}\right]=0;
		\end{align*}
		\begin{align*}
		\Real \skpLzwei{A v}{ -\im A v}=\Real \left[\im \norm{A v}_\Lzwei^2\right]=0
		\end{align*}

By the maximal inequality for the Poisson stochastic integral, see Theorem 4.5 in \cite{Dirksen}, and  Lemma \ref{lem-EstimateItoDifferencesEnergy}, we obtain
\begin{align}\label{burkholderEstimate1}
\left(\E \left[ \sup_{s\in[0,t]} \left\vert I_2(s)\right\vert^q\right]\right)^\frac{1}{q}
&\lesssim\left(\E   \left(\int_0^t\int_{\left\{\vert l \vert\le 1\right\}} \left\vert\energy(\groupBn u_n(s))-\energy(u_n(s))\right\vert^2 \nu(\df l)\df s\right)^{\frac{q}{2}}\right)^\frac{1}{q}\nonumber\\
&\quad +\left(\E   \int_0^t\int_{\left\{\vert l \vert\le 1\right\}} \left\vert\energy(\groupBn u_n(s))-\energy(u_n(s))\right\vert^q \nu(\df l)\df s\right)^\frac{1}{q}\nonumber\\
&\hspace{-1truecm}\lesssim\left(\E   \left(\int_0^t\int_{\left\{\vert l \vert\le 1\right\}} \vert l \vert^2 \left(\norm{u_n(s)}_\EA^2+\norm{u_n(s)}_{L^{\alpha+1}}^{\alpha+1}\right)^{2}\nu(\df l)\df s\right)^{\frac{q}{2}}\right)^\frac{1}{q}\nonumber\\
&+ \left(\E   \int_0^t\int_{\left\{\vert l \vert\le 1\right\}} \vert l \vert^q \left(\norm{u_n(s)}_\EA^2+\norm{u_n(s)}_{L^{\alpha+1}}^{\alpha+1}\right)^{q}\nu(\df l)\df s\right)^\frac{1}{q}.\nonumber\\
\end{align}
		We introduce the abbreviation
		\begin{align*}
		X_n:=\frac{1}{2}\norm{u_n}_{L^2}^2+\energy(u_n)
		\end{align*}
		and observe
		\begin{align}\label{Xdefocusing}
		\norm{u_n}_\EA^2+\norm{u_n}_{L^{\alpha+1}}^{\alpha+1}\lesssim X_n.
		\end{align}
		Moreover, we have
		\begin{align}\label{highNormsOverBall}
		\int_{\left\{\vert l \vert\le 1\right\}} \vert l\vert^q\,\nu(\df l)\le \int_{\left\{\vert l \vert\le 1\right\}} \vert l\vert^2\,\nu(\df l)<\infty,\qquad q\ge 2.
		\end{align}
%
%
		Thus, we can conclude
		\begin{align}\label{burkholderEstimate1}
		\left(\E \left[ \sup_{s\in[0,t]} \left\vert I_2(s)\right\vert^q\right]\right)^\frac{1}{q}
		&\lesssim\left(\E   \left(\int_0^t X_n(s)^2\df s\right)^{\frac{q}{2}}\right)^\frac{1}{q}+ \left(\E   \int_0^t X_n(s)^q\df s\right)^\frac{1}{q}\nonumber\\
		&= \norm{X_n}_{L^q(\Omega,L^2(0,t))}+\norm{X_n}_{L^q(\Omega,L^q(0,t))},\;\; t \in [0,T].
		\end{align}
		By Lemma \ref{lem-EstimateItoDifferencesEnergy} b), \eqref{Xdefocusing} and the Minkowski inequality
		\begin{align*}
		\left(\E \Big[\sup_{s\in [0,t]}\vert I_3(s)\vert^q\Big]\right)^\frac{1}{q}\lesssim& \int_{\left\{\vert l \vert\le 1\right\}} \vert l\vert^2 \nu(\df l) \left(\E\left(\int_0^t \left(\norm{u_n(r)}_\EA^2+\norm{u_n(r)}_{L^{\alpha+1}}^{\alpha+1}\right) \df r\right)^q\right)^\frac{1}{q}\\
&\hspace{-2truecm}
		\lesssim \int_{\left\{\vert l \vert\le 1\right\}} \vert l\vert^2 \nu(\df l) \int_0^t \left\Vert X_n(r)\right\Vert_{L^q(\Omega)} \df r
		\lesssim \int_0^t \left\Vert X_n\right\Vert_{L^q(\Omega,L^\infty(0,r))} \df r ,\;\; t \in [0,T].
		\end{align*}
Therefore,  from \eqref{ItoEnergyStart} and the previous estimates we get
\begin{align}\label{ItoEnergyAfterBurkholder}
\norm{X_n}_{L^q(\Omega,L^\infty(0,t))}
&\le \frac{1}{2}\norm{\widetilde{u_{0,n}}}_\Lzwei^2+\energy(\widetilde{u_{0,n}})+\left(\E \Big[ \sup_{s\in[0,t]} \vert I_2(s)\vert^q\Big]\right)^\frac{1}{q}+\left(\E \Big[ \sup_{s\in[0,t]} \vert I_3(s)\vert^q\Big]\right)^\frac{1}{q}\nonumber\\
&\lesssim \frac{1}{2}\norm{u_0}_\Lzwei^2+\energy(\widetilde{u_{0,n}})+	\norm{X_n}_{L^q(\Omega,L^2(0,t))}+	\norm{X_n}_{L^q(\Omega,L^q(0,t))}\nonumber\\	&\quad+\int_0^t \norm{X_n}_{L^q(\Omega,L^\infty(0,s))}\df s,\;\; t \in [0,T].	
\end{align}
Using Lemma $\ref{LemmaYOmegaNachLzweiZeit}$ with  $\varepsilon>0$ to estimate  $\norm{X_n}_{L^q(\Omega,L^2(0,t))}$ and $\norm{X_n}_{L^q(\Omega,L^q(0,t))},$ we get for $\;\; t \in [0,T]$,
\begin{align*}
\norm{X_n}_{L^q(\Omega,L^\infty(0,t))}
&\lesssim \frac{1}{2}\norm{u_0}_\Lzwei^2+\energy(\widetilde{u_{0,n}})+\varepsilon\norm{X_n}_{L^q(\Omega,L^\infty(0,t))}+	\int_0^t \norm{X_n}_{L^q(\Omega,L^\infty(0,s))}\df s.
\end{align*}
Taking $\eps$ sufficiently small we
end up with
\begin{align*}
\norm{X_n}_{L^q(\Omega,L^\infty(0,t))}
&\lesssim \frac{1}{2}\norm{u_0}_\Lzwei^2+\energy(\widetilde{u_{0,n}})+	\int_0^t \norm{X_n}_{L^q(\Omega,L^\infty(0,s))}\df s,\;\; t \in [0,T].	
\end{align*}
Finally, the Gronwall Lemma yields
\begin{align*}
\norm{X_n}_{L^q(\Omega,L^\infty(0,t))}\le C\left(\frac{1}{2}\norm{u_0}_\Lzwei^2+\energy(\widetilde{u_{0,n}})\right) e^{C t},\qquad t\in [0,T],
\end{align*}
where the constant $C=C(b_\EA,b_{\alpha+1},q, \alpha,F)>0$ is uniform in $n\in\N.$ As a consequence of \eqref{initialValueBoundedness} and Proposition \ref{PaleyLittlewoodLemma}, we obtain 
\begin{align}\label{initialEnergy}
\energy(\widetilde{u_{0,n}})&\lesssim \frac{\norm{u_0}_\Lzwei^2}{\norm{S_n u_0}_\Lzwei^2} \norm{\sqrtA S_n u_0}_\Lzwei^2+\frac{\norm{u_0}_\Lzwei^{\alpha+1}}{\norm{S_n u_0}_\Lzwei^{\alpha+1}}\norm{S_n u_0}_{L^{\alpha+1}}^{\alpha+1}\nonumber\\
&\lesssim \norm{\sqrtA u_0}_\Lzwei^2+\norm{u_0}_{L^{\alpha+1}}^{\alpha+1}\lesssim \energy(u_0)
\end{align}
for $n\ge n_0$ and $\energy(\widetilde{u_{0,n}})=0$ for $n<n_0.$ This completes the proof of Proposition \ref{EstimatesGalerkinSolution} a).\\

		\emph{Ad b):} Now, we continue with the proof of the Aldous condition. Let us fix $n\in \mathbb{N}$. We have  for all $t\in [0,T]$, almost surely
		\begin{align*}
		u_n(t) - \widetilde{u_{0,n}} =&-\im \int_0^t A u_n(s) \df s-\im \int_0^t P_n F(u_n(s)) \df s \\
		&+ \int_{0}^{t}\! \int_{\left\{\vert l \vert\le 1\right\}} \!\left[\groupBn u_n(s-) - u_n(s-)\right]\, \tilde{\eta}(\df s,\df l)\nonumber\\
		&+ \int_{0}^{t}\! \int_{\left\{\vert l \vert\le 1\right\}} \! \left\{ \groupBn u(s) - u(s) + \im \Bn(l) u(s)\right\}\, \nu(\df l)\df s\\
		=&:J_1(t)+J_2(t)+J_3(t)+J_4(t)
		\end{align*}
		in $H_n$. Let us next fix a sequence $\left(\tau_n\right)_{n\in\N}$ of stopping times and $\theta>0.$ By the above we infer that
		\begin{align*}
		\Vert u_n((\tau_n+\theta)\land T)-u_n(\tau_n)\Vert_\EAdual\le\sum_{k=1}^4 \Vert J_k((\tau_n+\theta)\land T)-J_k(\tau_n)\Vert_{\EAdual}.
		\end{align*}
		Hence, 		for a fixed $\eta>0$, we get
		\begin{align}\label{AldousStartingEstimate}
		\Prob \left\{\Vert u_n((\tau_n+\theta)\land T)-u_n(\tau_n)\Vert_\EAdual\ge \eta \right\}\le \sum_{k=1}^4 \Prob \left\{\Vert J_k((\tau_n+\theta)\land . T)-J_k(\tau_n)\Vert_{\EAdual}\ge \frac{\eta}{4}\right\}
		\end{align}
We aim to apply the Chebyshev  inequality and  estimate the expected value of each term in the sum on the RHS of \eqref{AldousStartingEstimate}. We use part a) for
		\begin{align*}
		\E\Vert J_1((\tau_n+\theta)\land T)-J_1(\tau_n)\Vert_{\EAdual}&\le \E \int_{\tau_n}^{(\tau_n+\theta)\land T} \Vert A u_n(s)\Vert_\EAdual \df s
		\le   \E\int_{\tau_n}^{(\tau_n+\theta)\land T} \Vert \sqrtA u_n(s)\Vert_\Lzwei \df s\\
&\lesssim \theta \E \big[ \sup_{s\in[0,T]} \norm{u_n(s)}_\EA\big]
\le \theta \E \big[ \sup_{s\in[0,T]} \norm{u_n(s)}_\EA^2\big]^{\frac{1}{2}}\le \theta C_1;
		\end{align*}
		the embedding $ \LalphaPlusEinsDual\hookrightarrow \EAdual$ and the nonlinear estimates $\eqref{nonlinearityEstimate}$ and $\eqref{boundantiderivative}$ for
		\begin{align*}
		\E\Vert J_2((\tau_n+\theta)\land T)&-J_2(\tau_n)\Vert_{\EAdual}\le \E \int_{\tau_n}^{(\tau_n+\theta)\land T} \Vert P_n F(u_n(s))\Vert_\EAdual \df s\\
		&\le \E \int_{\tau_n}^{(\tau_n+\theta)\land T} \Vert  F(u_n(s))\Vert_\EAdual \df s\lesssim  \E\int_{\tau_n}^{(\tau_n+\theta)\land T} \Vert  F(u_n(s))\Vert_\LalphaPlusEinsDual \df s\\&\lesssim  \E\int_{\tau_n}^{(\tau_n+\theta)\land T} \Vert  u_n(s)\Vert_\LalphaPlusEins^\alpha \df s
		\lesssim \theta \E \big[\sup_{s\in[0,T]} \norm{u_n(s)}_\EA^\alpha\big]
		\le \theta C_2
		\end{align*}
By the Levy-It\^o-isometry, Lemma \ref{AldousDifferences}, \eqref{LevyMeasure} and Proposition \ref{massConservationGalerkin} we get
		\begin{align*}
		\E\Vert J_3((\tau_n+\theta)\land T)-J_3(\tau_n)\Vert_{\EAdual}^2&
		\lesssim \E \left\Vert\int_{\tau_n}^{(\tau_n+\theta)\land T}\! \int_{\left\{\vert l \vert\le 1\right\}} \!\left[\groupBn u_n(s-) - u_n(s-)\right]\, \tilde{\eta}(\df s,\df l)\right\Vert_\Lzwei^2\\
		&= \E \int_{\tau_n}^{(\tau_n+\theta)\land T}\! \int_{\left\{\vert l \vert\le 1\right\}} \!\norm{\groupBn u_n(s) - u_n(s)}_\Lzwei^2\, \nu(\df l)\df s\\
		&\le b_\Lzwei \int_{\left\{\vert l \vert\le 1\right\}} \vert l\vert^2\nu(\df l) \E \int_{\tau_n}^{(\tau_n+\theta)\land T} \norm{u_n(s)}_\Lzwei^2\df s\lesssim \theta \norm{u_0}_\Lzwei^2
		\end{align*}
and
		\begin{align*}
		\E\Vert J_4((\tau_n+\theta)\land T)-&J_4(\tau_n)\Vert_{\EAdual}\\
		&= \E \left\Vert \int_{\tau_n}^{(\tau_n+\theta)\land T}\int_{\left\{\vert l \vert\le 1\right\}} \! \left\{ \groupBn u_n(s) - u_n(s) + \im \Bn(l) u_n(s)\right\}\, \nu(\df l) \df s \right\Vert_\EAdual\\
				&\lesssim  \E  \int_{\tau_n}^{(\tau_n+\theta)\land T}\int_{\left\{\vert l \vert\le 1\right\}} \left\Vert  \groupBn u_n(s) - u_n(s) + \im \Bn(l) u_n(s)\right\Vert_\Lzwei \nu(\df l) \df s \\		&\le \frac{1}{2} b_\Lzwei \int_{\left\{\vert l \vert\le 1\right\}} \vert l\vert^2\nu(\df l) \E \int_{\tau_n}^{(\tau_n+\theta)\land T} \norm{u_n(s)}_\Lzwei\df s\lesssim \theta \norm{u_0}_\Lzwei.	
		\end{align*}
		By the Chebyshev inequality, we obtain for a given $\eta>0$
		\begin{align}\label{AldousEins}
		\Prob \left\{\Vert J_k((\tau_n+\theta)\land T)-J_k(\tau_n)\Vert_{\EAdual}\ge \frac{\eta}{4}\right\}\le \frac{4}{\eta} \E\Vert J_k((\tau_n+\theta)\land T)-J_k(\tau_n)\Vert_{\EAdual}\le \frac{ 4C_k \theta}{\eta}
		\end{align}
		for $k\in \{1,2,4\}$ and
		\begin{align}\label{AldousZwei}
		\Prob \left\{\Vert J_3((\tau_n+\theta)\land T)-J_3(\tau_n)\Vert_{\EAdual}\ge \frac{\eta}{4}\right\}\le \frac{16}{\eta^2} \E\Vert J_3((\tau_n+\theta)\land T)-J_3(\tau_n)\Vert_{\EAdual}^2\le \frac{16 C_4 \theta}{\eta^2}.
		\end{align}
		Let us fix $\varepsilon>0.$ Due to  estimates $\eqref{AldousEins}$ and $\eqref{AldousZwei}$ we can choose $\delta_1,\dots,\delta_4>0$ such that
		\begin{align*}
		\Prob \left\{\Vert J_k((\tau_n+\theta)\land T)-J_k(\tau_n)\Vert_{\EAdual}\ge \frac{\eta}{4}\right\}\le \frac{\varepsilon}{4}
		\end{align*}
		for $0<\theta\le \delta_k$ and $k=1,\dots,4.$ With $\delta:= \min \left\{\delta_1,\dots,\delta_4\right\},$  using $\eqref{AldousStartingEstimate}$ we get
		\begin{align*}
		\Prob \left\{\Vert J_k((\tau_n+\theta)\land T)-J_k(\tau_n)\Vert_{\EAdual}\ge \eta\right\}\le \varepsilon
		\end{align*}
		for all $n\in\N$ and $0<\theta\le \delta$ and therefore, the Aldous condition $[A]$ holds in $E_A^*.$
\end{proof}

We continue with the a priori estimate for solutions of $\eqref{marcusGalerkin}$ with a focusing nonlinearity. Note that this case is harder since the expression
\begin{align*}
\frac{1}{2}\norm{v}_{H}^2+\energy(v)=\frac{1}{2}\norm{v}_\EA^2+\Fhat(v), \qquad v\in H_n,
\end{align*}
does not dominate $\norm{v}_\EA^2,$ because $\Fhat$ is negative. Nevertheless, we will see that the $\EA$-norm is still the dominating part under the additional Assumption
\ref{focusing} i'), which leads to a restriction to the maximal degree of the nonlinearity $F.$ In particular, uniform estimates in $\EA$ are still possible.

\begin{Prop}\label{EstimatesGalerkinSolutionFocusing}
	Under Assumption $\ref{focusing}$ i'), the following assertions hold:
	\begin{enumerate}
		\item[a)] For all $r\in[1,\infty),$ there is a constant 
		\begin{align*}
			C=C(\norm{u_0}_{H},\norm{\sqrtA u_0}_\Lzwei,\norm{u_0}_{L^{\alpha+1}},\gamma,\alpha,T, F,b_\EA,b_{\alpha+1},r)>0
		\end{align*}
		 with
		\begin{align*}
		\sup_{n\in\N}\E \Big[\sup_{t\in[0,T]} \norm{u_n(t)}_\EA^{r}\Big]\le C;
		\end{align*}
		\item[b)] The sequence $(u_n)_{n\in\N}$ satisfies the Aldous condition $[A]$ in $\EAdual.$
	\end{enumerate}
	In particular, the sequence $\left({\Prob}^{u_n}\right)_{n\in\N}$ is tight in $Z_T$ by Proposition \ref{TightnessCriterion}.
\end{Prop}	

\begin{proof}	
	\emph{ad a):}
	Let $\varepsilon>0.$ Assumption $\ref{focusing}$ i') and Young's inequality imply that  there are $\gamma>0$ and $C_\varepsilon>0$ such that
	\begin{align}\label{focusingNonlinearityControl}
	\norm{u}_\LalphaPlusEins^{\alpha+1} \lesssim \varepsilon \norm{u}_\EA^2+C_\varepsilon \norm{u}_H^\gamma,\qquad u\in\EA,
	\end{align}
	and therefore by Proposition $\ref{massConservationGalerkin}$, we infer that
	\begin{align}\label{gagliardoFhat}
	-\Fhat(u_n(s))&\lesssim \norm{u_n(s)}_\LalphaPlusEins^{\alpha+1} \lesssim \varepsilon \norm{u_n(s)}_\EA^2+C_\varepsilon \norm{u_n(s)}_H^\gamma\nonumber\\
	&\lesssim \varepsilon \norm{\sqrtA u_n(s)}_{H}^2+\varepsilon \norm{u_0}_H^2+C_\varepsilon \norm{u_0}_H^\gamma,\qquad s\in [0,T].
	\end{align}
	By analogous calculations as in the proof of Proposition $\ref{EstimatesGalerkinSolution}$ we get
			\begin{align}\label{ItoEnergyStartFocusing}
			\frac{1}{2}\norm{ \sqrtA u_n(s)}_H^2=&-\Fhat(u_n(s))+\energy\left(u_n(s)\right)\nonumber\\
			=&-\Fhat(u_n(s))+\energy\left(\widetilde{u_{0,n}}\right)\nonumber\\
			&+\int_0^s \int_{\left\{\vert l \vert\le 1\right\}} \left[\energy(\groupBn u_n(r-))-\energy(u_n(r-))\right] \tilde{\eta}(\df l,\df r)\nonumber\\
						&+\int_0^s \int_{\left\{\vert l \vert\le 1\right\}} \left[\energy(\groupBn u_n(r))-\energy(u_n(r))+\energy'[u_n(r)]\left(\im \Bn(l) u_n(s)\right)\right]  \nu(\df l)\df r\nonumber\\
						=:&-\Fhat(u_n(s))+\energy\left(\widetilde{u_{0,n}}\right)+I_1(s)+I_2(s)
			\end{align}
	almost surely for all $t\in [0,T].$
		We abbreviate
		\begin{align*}
		Y_n(s):=\norm{u_0}_{L^2}^2+\norm{\sqrtA u_n(s)}_{L^2}^2+\norm{u_n(s)}_{L^{\alpha+1}}^{\alpha+1},\qquad s\in [0,T].
		\end{align*}
		Let $q>2$ and recall $\eqref{highNormsOverBall}$ as well as the mass conservation from Proposition \ref{massConservationGalerkin}.		
		As in the proof of Proposition $\ref{EstimatesGalerkinSolution}$, we estimate
			\begin{align}\label{initialEnergyFocusing}
		\vert\energy(\widetilde{u_{0,n}})\vert&\lesssim 
		\norm{\sqrtA u_0}_\Lzwei^2+\norm{u_0}_{L^{\alpha+1}}^{\alpha+1},
		\end{align}
		\begin{align}\label{FocusingI1}
		\left(\E \left[ \sup_{s\in[0,t]} \left\vert I_1(s)\right\vert^q\right]\right)^\frac{1}{q}
		&\lesssim \left(\int_{\left\{\vert l \vert\le 1\right\}} \vert l \vert^2 \nu(\df l)\right)^\frac{1}{2} \left(\E   \left(\int_0^t \left(\norm{u_n(s)}_\EA^2+\norm{u_n(s)}_{L^{\alpha+1}}^{\alpha+1}\right)^{2}\df s\right)^{\frac{q}{2}}\right)^\frac{1}{q}\nonumber\\
		&\quad +\left(\int_{\left\{\vert l \vert\le 1\right\}} \vert l \vert^q \nu(\df l)\right)^\frac{1}{q} \left(\E  \int_0^t \left(\norm{u_n(s)}_\EA^2+\norm{u_n(s)}_{L^{\alpha+1}}^{\alpha+1}\right)^{q}\df s\right)^\frac{1}{q}\nonumber\\
		&\lesssim\,\norm{Y_n}_{L^q(\Omega,L^2(0,t))}+\norm{Y_n}_{L^q(\Omega,L^q(0,t))};
		\end{align}
		\begin{align}\label{FocusingI2}
		\left(\E \Big[\sup_{s\in [0,t]}\vert I_2(s)\vert^q\Big]\right)^\frac{1}{q}&\lesssim\int_{\left\{\vert l \vert\le 1\right\}} \vert l\vert^2 \nu(\df l) \int_0^t \left\Vert\norm{u_n}_\EA^2+\norm{u_n}_{L^{\alpha+1}}^{\alpha+1}\right\Vert_{L^q(\Omega,L^\infty(0,r))} \df r\nonumber\\
		&\lesssim \int_0^t \norm{Y_n}_{L^q(\Omega,L^\infty(0,r))}\df r.
		\end{align}
		Using  \eqref{gagliardoFhat}, \eqref{initialEnergyFocusing}, \eqref{FocusingI1} and \eqref{FocusingI2}  in \eqref{ItoEnergyStartFocusing}, we obtain
		\begin{align*}
		\bigNorm{\norm{\sqrtA u_n}_{L^2}^2}_{L^q(\Omega,L^\infty(0,t))}\lesssim&\, \bigNorm{\norm{\sqrtA u_n}_{L^2}^2}_{L^q(\Omega,L^\infty(0,t))} \varepsilon +\varepsilon \norm{u_0}_{L^2}^2+C_\varepsilon \norm{u_0}_{L^2}^\gamma\\
		&+\norm{\sqrtA u_0}_{L^2}^2+\norm{ u_0}_{L^{\alpha+1}}^{\alpha+1}+\norm{Y_n}_{L^q(\Omega,L^2(0,t))}\\
		&+\norm{Y_n}_{L^q(\Omega,L^q(0,t))}+\int_0^t \norm{Y_n}_{L^q(\Omega,L^\infty(0,r))}\df r.
		\end{align*}
		If we employ Lemma \ref{LemmaYOmegaNachLzweiZeit} to estimate $\norm{Y_n}_{L^q(\Omega,L^2(0,t))}$ and $\norm{Y_n}_{L^q(\Omega,L^q(0,t))},$ we get
		\begin{align}\label{sqrtAFocusing}
		\bigNorm{\norm{\sqrtA u_n}_{L^2}^2}_{L^q(\Omega,L^\infty(0,t))}\lesssim&\,  \bigNorm{\norm{\sqrtA u_n}_{L^2}^2}_{L^q(\Omega,L^\infty(0,t))} \varepsilon+\varepsilon \norm{u_0}_{L^2}^2+C_\varepsilon \norm{u_0}_{L^2}^\gamma\nonumber\\
		&+\norm{\sqrtA u_0}_{L^2}^2+\norm{ u_0}_{L^{\alpha+1}}^{\alpha+1}+\varepsilon\norm{Y_n}_{L^q(\Omega,L^\infty(0,t))}\nonumber\\
		&+\int_0^t \norm{Y_n}_{L^q(\Omega,L^\infty(0,r))}\df r.
		\end{align}
				In order to estimate the terms with $Y_n$ by the LHS of \eqref{sqrtAFocusing}, we exploit \eqref{gagliardoFhat} to get		
				\begin{align*}
				\norm{Y_n}_{L^q(\Omega,L^\infty(0,t))}&\le  \norm{u_0}_H^{2}+\bigNorm{\norm{\sqrtA u_n}_{H}^{2}}_{L^q(\Omega,L^\infty(0,t))}+\bigNorm{\norm{ u_n}_{L^{\alpha+1}}^{\alpha+1}}_{L^q(\Omega,L^\infty(0,t))}\\
				&\le \left(1+\varepsilon \right)\bigNorm{\norm{\sqrtA u_n}_{H}^{2}}_{L^q(\Omega,L^\infty(0,t))}+	C(\varepsilon,\norm{u_0}_H).
				\end{align*}	
		Now, we choose $\varepsilon>0$ sufficiently small and end up with
		\begin{align*}
		\bigNorm{\norm{\sqrtA u_n}_{L^2}^2}_{L^q(\Omega,L^\infty(0,t))}\le C\left(1+\int_0^t \bigNorm{\norm{\sqrtA u_n}_{L^2}^2}_{L^q(\Omega,L^\infty(0,r))}\df r\right)
		\end{align*}
		for some $C=C(\norm{u_0}_{L^2},\norm{\sqrtA u_0}_\Lzwei,\norm{u_0}_{L^{\alpha+1}},\gamma,\alpha,T, F,b_\EA,b_{\alpha+1},q)$ independent of $n.$ From the Gronwall Lemma, we infer
		\begin{align}\label{prefinalEstimateFocusing}
		\bigNorm{\norm{\sqrtA u_n}_{L^2}^2}_{L^q(\Omega,L^\infty(0,t))}\le C e^{Ct},\qquad t\in [0,T].
		\end{align}

		In view of Proposition \ref{massConservationGalerkin}, we have proved the assertion for $r=2q> 4.$ The case $r\in [1,4]$ is an easy consequence of the H\"older inequality.

	\emph{ad b).} The proof of the Aldous condition is similar to the defocusing case, see Proposition \ref{EstimatesGalerkinSolution} b).
\end{proof}

\begin{corollary}\label{GalerkinTight}
	Under Assumption \ref{focusing}, the sequence $\left(u_n\right)_{n\in\N}$ of Galerkin solutions is tight on $Z_T.$
\end{corollary}

\begin{proof}
	Immediate consequence of Propositions \ref{TightnessCriterion}, \ref{EstimatesGalerkinSolution} and \ref{EstimatesGalerkinSolutionFocusing}.
\end{proof} 
\section{Construction of a martingale solution}

%
%

In this section, we will use the compactness results and the uniform estimates from the previous sections to complete the proof of Theorem \ref{mainTheorem}.
Let us recall
\begin{align*}
Z_T:=\cadlagHminuseins\cap \LalphaPlusEinsAlphaPlusEins \cap \weaklyCadlag{\EA}.
\end{align*}
The first step is to prove that Proposition \ref{SkohorodJakubowski} can be applied with
\begin{align*}
\mathcal{X}_1:=\countingMeasures,\qquad \mathcal{X}_2:=Z_T.
\end{align*}
Here, $\countingMeasures$ denotes the set of all $\bar{\N}$-valued Borel measures $\xi$ on $[0,T]\times \R^N$ with $\xi(\mathcal{S}_n)<\infty$ for all $n\in \N,$ for some sequence  $\mathcal{S}_n\subset [0,T]\times \R^N$ of Borel sets with $\mathcal{S}_n \uparrow [0,T]\times \R^N$ and
$\Leb \otimes \nu(\mathcal{S}_n)<\infty$ for all $n\in\N.$
It is well known, see e.g. Lemma 2.53 in the second authors dissertation \cite{FHornungPhD} or Section 1 in \cite{budhiraja2011variational}, that $\countingMeasures$ is a complete separable metric space. 

Moreover, we determine the $\sigma$-algebra $\mathcal{A}.$ Of course, it would be natural to equip $Z_T$ with the Borel $\sigma$-algebra $\B(Z_T),$ but it turns out that $\mathcal{A}$ is strictly contained in $\B(Z_T).$
Given real-valued functions $f_m$ on a topological space $Z$, we will frequently use the notation $f=\left(f_1,f_2,\dots\right)$ and the fact that $\sigma(f_{m}: m\in\N)=f^{-1}(\B(\R^\infty)),$ where $\R^\infty$ is equipped with the locally convex topology induced by the seminorms $p_k(x):=\vert x_k\vert.$

\begin{lemma}\label{BorelSetsInCoarseTop}
	Let $X$ be a set and $f_m: X\to \R,$ $m\in\N.$ Let $\mathcal{O}_X$ be the coarsest topology such $f_m$ is continuous for all $m\in\N.$ Then, we have
	\begin{align*}
	\B(X):=\sigma(\mathcal{O}_X)=\sigma(f_m: m\in\N).
	\end{align*}
\end{lemma}

\begin{proof}
	The direction $"\supset"$ is obvious by the continuity of $f_{m}$ for $m\in\N.$ In view of the good set principle, it is sufficient for the other inclusion to show that each $O\in \mathcal{O}_X$ is contained in the $f^{-1}(\B(\R^\infty)).$ Since each $O\in \mathcal{O}_X$ is of the form
	\begin{align*}
	O=\bigcup_{i\in I} \bigcap_{k=1}^K f_3^{-1}(O_{i,k}),\qquad \text{$O_{i,k}$ open in $\R^\infty,$}
	\end{align*}
	see \cite{Folland}, Proposition 4.4,
	we can write represent $O$ as the inverse image of the open set $\bigcup_{i\in I} \bigcap_{k=1}^K O_{i,k}$ under the continuous function $f,$ which verifies the assertion.
\end{proof}

\begin{lemma}\label{LemmaZTSkorohodJakubowski}
	There is a countable family $F$ of real-valued continuous functions on $Z_T$ that separates points of $Z_T$ and generates the $\sigma$-algebra
	\begin{align}\label{SigmaAlgebraGeneratedByF}
	\mathcal{A}=\sigma\left(\B(Z_1\cap Z_2)|_{Z_T}\cup \sigma(F_3) \right),
	\end{align}
	where $F_3$ consists of real-valued continuous functions on $Z_3$  separating points of $Z_3.$
\end{lemma}

\begin{proof}
	
	\emph{Step 1.} For each $Z_i$, we give a sequence  $\left(f_{m,i}\right)_{m\in\N}$ of continuous functions $f_{m,i}: Z_i\to \R$ separating points and determine the generated $\sigma$-algebras.
	
	Let $\left\{\varphi_k: k\in \N \right\}$ be a sequence with $\norm{\varphi_k}_{\EA}\le 1$ and $\norm{x}_{\EAdual}=\sup_{k\in\N}\vert\Real \duality{x}{\varphi_k}\vert$ for all $x\in \EAdual$
	and $\left\{t_l: l\in \N \right\}$ be dense in $[0,T].$  We set
	\begin{align*}
	f_{k,l,1}(u):=\Real \duality{u(t_l)}{\varphi_k},\qquad u\in Z_1,\quad k,l\in\N
	\end{align*}
	and for $n\in\N,$  we denote
	\begin{align*}
	\pi_{t_1,\dots,t_n}: Z_1\to (\EAdual)^n,\qquad u\mapsto \left(u(t_1),\dots, u(t_n)\right).
	\end{align*}
	From \cite{JakubowskiSkorohodTopology}, Corollary 2.4, we know that
	\begin{align*}
	\B(Z_1)=\sigma(\pi_{t_1,\dots,t_n}:n\in\N).
	\end{align*}
	But since $\pi_{t_1,\dots,t_n}$ is strongly measurable in $(\EAdual)^n$ if and only if
	\begin{align*}
	Z_1\ni u\mapsto \Real \duality{\pi_{t_1,\dots,t_n}(u)}{(\varphi_{k_1},\dots, \varphi_{k_n})}_{(\EAdual)^n,(\EA)^n}=\sum_{j=1}^{n} f_{k_j,j,1}(u)
	\end{align*}
	for all $k_1,\dots,k_n\in \N,$
	we obtain $\B(Z_1)=\sigma(f_{k,l}:k\in\N,l\in\N).$ By right-continuity and the choice of $\varphi_k,$ $k\in\N,$ the $f_{k,l}$ separate points in $Z_1$ and they are continuous since convergence in $Z_1$ implies pointwise convergence.
	
	The existence of $\left(f_{m,2}\right)_{m\in\N}$ is a consequence of the Hahn-Banach-Theorem in $Z_2.$ For the details, we refer to \cite{FHornungPhD}, Lemma 2.28. 
	Let $\left\{h_k: k\in \N \right\}$ and $\left\{t_l: l\in \N \right\}$ be dense subsets of $\EAdual$ and $[0,T],$ respectively. We set
	\begin{align*}
	f_{k,l,3}(u):=\Real \duality{u(t_l)}{h_k},\qquad u\in Z_3,\quad k,l\in\N.
	\end{align*}
	and denote the enumeration of $\left(f_{k,l,3}\right)_{k,l\in\N}$ by $\left(f_{m,3}\right)_{m\in\N}.$ By the definition of the topology in $Z_3$ and the fact that convergence in $\D([0,T])$ implies pointwise convergence, we obtain that $f_{m,3}$ is continuous. Suppose that $f_{m,3}(u_1)=f_{m,3}(u_2)$ for all $u_1,u_2\in Z_3.$ From the right-continuity of $[0,T]\ni t\mapsto \Real \duality{u_j(t)}{h_k}$ and the density of $\left(t_l\right)_l$ $\left(h_k\right)_k$, we infer $u_1(t)=u_2(t)$ for all $t\in \N,$ i.e. $\left(f_{m,3}\right)_{m\in\N}$ separates points in $Z_3.$

%
%
	
	\emph{Step 2.} We define $F_j:=\left\{f_{m,j}|_{Z_T}: m\in\N\right\}$ and set $\mathcal{A}:= \sigma(F),$ where $F:=F_1\cup F_2\cup F_3.$
	%
	We would like to prove \eqref{SigmaAlgebraGeneratedByF}.
	Above, we obtained $\sigma(f_{m,j}:m\in\N)=\B(Z_j)$ for $j=1,2.$ Since we have
	\begin{align*}
	\sigma\left(f_{m,j}|_{Z_1\cap Z_2}:m\in\N\right)=\sigma(f_{m,j}:m\in\N)|_{Z_1\cap Z_2}
	\end{align*}
	and
	\begin{align*}
	\B(Z_1\cap Z_2)=\sigma\left(\bigcup_{j=1,2}\B(Z_j)|_{Z_1\cap Z_2}\right),
	\end{align*}
	we conclude
	\begin{align*}
	\B(Z_1\cap Z_2)&=\sigma\left(\bigcup_{j=1,2}\sigma\left(f_{m,j}|_{Z_1\cap Z_2}:m\in\N\right)\right)
	\end{align*}
	and thus,
	\begin{align*}
	\B(Z_1\cap Z_2)|_{Z_T}
	&=\sigma\left(f_{m,1}|_{Z_T},f_{m,2}|_{Z_T}:m\in\N\right)=\sigma(F_1\cup F_2).
	\end{align*}
	Similarly, we obtain $\mathcal{A}=\sigma\left(\B(Z_1\cap Z_2)|_{Z_T}\cup \sigma(F_3) \right).$
	
\end{proof}

\begin{remark}
	By Lemma $\ref{BorelSetsInCoarseTop},$ we have $\sigma\left(f_{m,3}:m\in\N\right)=\sigma(\tilde{\mathcal{O}}_{Z_3}),$ where $\tilde{\mathcal{O}}_{Z_3}$ is the coarsest topology such that $f_{m,3}$ is continuous for all $m\in\N.$ In particular, we have
	\begin{align*}
	\sigma\left(f_{m,3}:m\in\N\right)\subsetneq \B(Z_3),
	\end{align*}
	since convergence in $\D([0,T])$ implies pointwise convergence, but not vice versa. In particular, we would get $\mathcal{A}=\B(\tilde{Z}_T)$ where $\tilde{Z}_T$ is the topological space arising when we replace the topology on $Z_3$ by $\tilde{\mathcal{O}}_{Z_3}.$
\end{remark}

By the previous Lemma and the uniform estimates from Propositions \ref{EstimatesGalerkinSolution} and \ref{EstimatesGalerkinSolutionFocusing}, we can apply Proposition \ref{SkohorodJakubowski} to the sequence $\left(u_n\right)_{n\in\N}$ of Galerkin solutions. As a result, we obtain a candidate $v$ for the martingale solution.

\begin{corollary}\label{cor-main} Let $\left(u_n\right)_{n\in\N}$ be the sequence of solutions to the Galerkin equation $\eqref{marcusGalerkin}$ on $(\Omega,\F,\Prob)$ and $\mathcal{A}$ be the $\sigma$-algebra on $Z_T$ defined in \eqref{SigmaAlgebraGeneratedByF}.
	\begin{itemize}
		\item[a)] There are a probability space $(\bar{\Omega},\bar{\F},\tildeProb),$ a subsequence $\left(u_{n_k}\right)_{k\in\N}$ and random variables $v, v_k: \bar{\Omega} \to Z_T$ and  $\bar{\eta}_k, \bar{\eta}: \bar{\Omega}\to \countingMeasures$ with
		\begin{enumerate}
			\item[i)] $\tildeProb^{(\bar{\eta}_k,v_k)}=\Prob^{(\eta,u_{n_k})}$ for $k\in\N,$
			\item[ii)] $(\bar{\eta}_k,v_k) \to \left(\bar{\eta},v\right)$ in $\countingMeasures\times Z_T$ almost surely for $k\to \infty,$
			\item[iii)] $\bar{\eta}_k=\bar{\eta}$ almost surely.
		\end{enumerate}
		Moreover, $\bar{\eta}_k, \bar{\eta}$ are time-homogeneous Poisson random measures on $[0,T]\times \R^N$ with intensity measure $\Leb\otimes \nu.$ w.r.t to the filtration $\bar{\Filtration}$ defined by the augmentation of
		\begin{align*}
		\bar{\F}_t:=\sigma\left(\bar{\eta}_k(s), v_m(s), v(s):  k\in\N,m\in\N, s\in [0,t]\right),
		\end{align*}
		where by the notation ${\bar\eta_k}(s)$ we mean all random variables of the form $
{\bar\eta_k}((0,s] \times B_1)$,
where $B_1$ is a measurable set in $B$.
		\item[b)]  We have $v_k \in  \D\left([0,T],H_k\right)$ $\bar{\Prob}$-a.s. and for all $r\in [1,\infty),$ there is
		$C=C( T, \norm{u_0}_\EA,r)>0$
		with
		\begin{align*}
		\sup_{k\in\N} \Etilde \left[ \norm{v_k}_\LinftyHeins^r\right]\le C.
		\end{align*}
		\item[c)] For all $r\in [1,\infty),$ we have
		\begin{align*}
		\Etilde \left[ \norm{v}_\LinftyHeins^r\right]\le C
		\end{align*}
		with the same constant $C>0$ as in $b).$
	\end{itemize}
	
\end{corollary}	

\begin{remark}
	The fact that for each $n\in\N$, $u_n$ is an  $(Z_T,\mathcal{A})$-valued  random variable   is true since $\D([0,T],H_n)\subset Z_j$ for each $n\in\N$ and each $j=1,2,3$, see \eqref{eqn-Z_T} for the definition of the spaces $Z_j$,  with continuity of the canonical embedding. In particular
	\begin{align*}
		\left\{B\cap \D([0,T],H_n): B\in \B(Z_T)\right\}&=\sigma\left(\left\{B\cap \D([0,T],H_n): \text{$B$ closed in $Z_T$}\right\}\right)\\
		&\subset \sigma(\left\{\tilde{B}: \text{$\tilde{B}$ closed in $\D([0,T],H_n)$}\right\})=\B(\D([0,T],H_n)).
	\end{align*}
	Since $u_n$ is random variable in $\D([0,T],H_n),$ we infer that
	\begin{align*}
		\left\{u_n\in B\right\}=\left\{u_n\in B\cap \D([0,T],H_n)\right\}\in \F
	\end{align*}
	for all $B\in \mathcal{A}.$
\end{remark}

\begin{proof}
	\emph{ ad a).}
	We apply Proposition \ref{SkohorodJakubowski} with
	\begin{align*}
	\mathcal{X}_1:=\countingMeasures, \qquad \mathcal{X}_2:=Z_T
	\end{align*}
	and $\chi_n=\left(\eta, u_n\right),$ $n\in\N.$ The tightness of $\chi_n$ is guaranteed by Corollary \ref{GalerkinTight} and the fact that random variables on metric spaces are tight, see \cite{parthasaraty}, Theorem 3.2. In Lemma \ref{LemmaZTSkorohodJakubowski}, we have checked that $Z_T$ fulfills the assumptions of Proposition \ref{SkohorodJakubowski} with the $\sigma$-algebra $\mathcal{A}$ from above. 	
	For the proof of the last assertion, we refer to \cite{BrzezniakHausenblasReactionDiffusion}, Section 8, Step III.

	\emph{ad b).}  Since $\D\left([0,T],H_k\right)$ is contained in $Z_j$ for $j=1,\dots, 4,$   the definition of $\mathcal{A}$ yields that $\D\left([0,T],H_k\right)\in \mathcal{A}.$  Hence, we obtain $v_k \in  \D\left([0,T],H_k\right)$ $\bar{\Prob}$-a.s. as an immediate consequence of the identity of the laws of $v_k$ and $u_{n_k}.$

	 The uniform estimate follows from the respective estimates for $\left(u_{n_k}\right)_{k\in\N},$ see Propositions \ref{EstimatesGalerkinSolution} and \ref{EstimatesGalerkinSolutionFocusing}, via the identity of laws, since $\D\left([0,T],H_k\right)\ni w\mapsto \sup_{t\in [0,T]}\norm{w(t)}_\EA$ is a measurable function.
	
	\emph{ad c).} We can follow the lines of the proof of Proposition 6.1 c) in \cite{ExistencePaper}.
\end{proof}	

\begin{corollary}\label{cor-main2} In the framework of Corollary \ref{cor-main}, we have
 $\bar{\mathbb{P}}$-almost surely, for each $k\in \mathbb{N}$, 
\[
	\norm{v_k(t)}_H=\norm{u_0}_H \mbox{ for all } t\in [0,T]. 
\]
\end{corollary}
\begin{proof}
Let us fix $k\in \mathbb{N}.$ Then, the set 
\begin{align*}
	\mathcal{S}=\bigl\{u \in \mathbb{D}([0,T],H_{n_k}): \norm{u(t)}_H=\norm{u_0}_H \mbox{ for all } t\in [0,T] \bigr\}
\end{align*}
is closed in $\mathbb{D}([0,T],H_{n_k})$ by Corollary \ref{cor-App C-02}. Therefore, $\mathcal{S}$ is a Borel set in $Z_T.$  By Corollary \ref{cor-main}, the laws of $v_k$ and $u_{n_k}$ are equal.  Since by Proposition \ref{massConservationGalerkin} the law of $u_{n_k}$  is concentrated on $\mathcal{S},$ so is the law of $v_k$. The proof is thus complete.
\end{proof}

It remains to show that $\left(\bar{\Omega},\bar{\F},\bar{\Prob},\bar{\eta},\bar{\Filtration},u\right)$ is indeed martingale solution. The compensated Poisson random measure induced by $\bar{\eta}$ is denoted by
$\tilde{\bar{\eta}}:=\bar{\eta}-\Leb \otimes\nu.$
 We need the following convergence results.

\begin{lemma}\label{LemmaConvergences}
	Let $\psi\in \EA.$ Then, we have the following convergences in  $L^2(\bar{\Omega}\times [0,T])$ as $n\to \infty:$
	\begin{align}\label{convergenceNonIntegralTerms}
		\Real \skpH{v_n-\widetilde{u_{0,n}}}{\psi}\to \Real \skpH{v-u_0}{\psi}
	\end{align}
	\begin{align}\label{convergenceDeterministicTerm}
		\int_0^\cdot\Real \skpH{ A v_n(s)+P_n F(v_n(s))}{\psi}\df s \to \int_0^\cdot\Real \duality{ A v(s)+ F(v(s))}{\psi}\df s;
		\end{align}
		\begin{align}\label{convergenceNoiseTerm}
		&\int_0^\cdot \int_{\left\{\vert l \vert\le 1\right\}} \!\Real \skpH{\groupBn v_n(s-) - v_n(s-)}{\psi}\, \tilde{\bar{\eta}}(\df s,\df l)\nonumber\\&\qquad\to \int_0^\cdot \int_{\left\{\vert l \vert\le 1\right\}} \!\Real \skpH{\groupB v(s-) - v(s-)}{\psi} \tilde{\bar{\eta}}(\df s,\df l);
		\end{align}
		\begin{align}\label{convergenceMarcusTerm}
			 \int_{0}^{\cdot}\! \int_{\left\{\vert l \vert\le 1\right\}} &\! \Real \skpH{\groupBn v_n(s) - v_n(s) + \im \Bn(l) v_n(s)}{\psi} \nu(\df l)\df s\nonumber\\&\hspace{4cm}\to \int_{0}^{\cdot}\! \int_{\left\{\vert l \vert\le 1\right\}} \! \Real \skpH{ \groupB v(s) - v(s) + \im \B(l) v(s)}{\psi} \nu(\df l)\df s.
		\end{align}	
\end{lemma}

\begin{proof}
	\emph{ad \eqref{convergenceNonIntegralTerms}.} We get \eqref{convergenceNonIntegralTerms} pointwise in $\bar{\Omega}\times [0,T]$ from \eqref{initialValueConvergence} and $v_n\to v$ in $L^2(0,T;H).$ In view of
	\begin{align*}
		\Etilde \int_0^T\vert \Real \skpH{v_n(t)-\widetilde{u_{0,n}}}{\psi}\vert^r \df t&\le \norm{\psi}_H^r 	\Etilde \int_0^T \left(\norm{v_n(t)}_H+\norm{u_0}_H\right)^r \df t\le  \norm{\psi}_H^r T 2^r \norm{u_0}_H^r <\infty
	\end{align*}
	for $r>2,$ Vitali's convergence Theorem yields the assertion.\\
	
	\emph{ad \eqref{convergenceDeterministicTerm}.} Let us fix $\omega\in \bar{\Omega}$ and $t\in [0,T].$ Then,
	\begin{align*}
		\int_0^t\Real \skpH{ P_n F(v_n(s))}{\psi}\df s \to \int_0^t\Real \duality{ F(v(s))}{\psi}\df s
	\end{align*}
	follows from $v_n\to v$ in $\LalphaPlusEinsAlphaPlusEins,$ see \cite{ExistencePaper}, Lemma 6.2, step 3. Moreover,
	\begin{align*}
		 \Real \duality{A (v_n(s)-v(s))}{\psi}= \Real \duality{v_n(s)-v(s)}{A\psi}\to 0
	\end{align*}
	for all $s\in [0,T]$ by $v_n\to v$ in $\D_w([0,T],\EA).$ Via
	\begin{align*}
		\Etilde \int_0^T \int_0^t \vert \Real \duality{A v_n(s)}{\psi}\vert^r \df s\df t\le \norm{\psi}_\EA^r T^2 \Etilde \Big[\sup_{s\in [0,T]} \norm{v_n(s)}_{\EA}^r\Big]<\infty,
	\end{align*}
	\begin{align*}
		\Etilde \int_0^T \left\vert \int_0^t \Real \skpH{ P_n F(v_n(s))}{\psi}\df s\right\vert^r\df t
		&\le T^{1+r} \norm{\psi}_\EA^r \Etilde \big[\sup_{s\in [0,T]} \norm{F(v_n(s))}_\EAdual^r\big]\\
		&\lesssim T^{1+r} \norm{\psi}_\EA^r \Etilde \big[\sup_{s\in [0,T]} \norm{v_n(s)}_\EA^{r\alpha}\big]<\infty
	\end{align*}
	for $r>2$, Vitali yields \eqref{convergenceDeterministicTerm} in  $L^2(\bar{\Omega}\times [0,T]).$\\
	
	\emph{ad \eqref{convergenceNoiseTerm}.} In view of the It\^o isometry, it is equivalent to prove
	\begin{align}
		\int_0^\cdot \int_{\left\{\vert l \vert\le 1\right\}} \vert \Real \skpH{\groupBn v_n(s) - v_n(s)-\left[\groupB v(s) - v(s)\right]}{\psi}\vert^2 \nu(\df l)\df s\to 0,\qquad n\to \infty,
	\end{align}
	in $L^1(\bar{\Omega}\times [0,T]).$
	For $x\in H,$ Lebesgue yields
	\begin{align*}
		\norm{\groupBn x-\groupB x}_H&=\bigNorm{\int_0^1 \frac{\df }{\df s}\left[\groupBnS e^{-\im (1-s)\B(l)}x\right]\df s}_H\\
			&\le \int_0^1 \norm{\left(\Bn(l)-\B(l)\right)\groupBnS e^{-\im (1-s)\B(l)}x }_H\df s\to 0,\qquad n\to \infty.
	\end{align*}
	From $v_n\to v$ almost surely in $L^2(0,T;H)$ and again Lebesgue, we infer
		\begin{align}\label{pointwiseConvergenceStochTerm}
		&\int_0^t\vert \Real \skpH{\groupBn v_n - v_n-\left[\groupB v - v\right]}{\psi}\vert ^2\df s\\
		&\qquad\le 2 \int_0^t \left(\norm{\groupBn \left(v - v_n\right)}_H^2+\norm {v_n-v}_H^2+\norm{\left[\groupBn-\groupB \right]v}_H \right)\norm{\psi}_H^2 \df s\to 0
		\end{align}
	as $n\to \infty$ almost surely for all $t\in [0,T]$ and $l\in B(0,1).$ Since we have
	\begin{align}\label{MajoranteInL}
		&\int_0^t\vert \Real \skpH{\groupBn v_n - v_n-\left[\groupB v - v\right]}{\psi}\vert^2 \df s\nonumber\\
		&\qquad\le 2 \norm{\psi}_H^{2} b_H \vert l\vert^2 \left(\norm{v_n}_{L^2(0,t;H)}^2+\norm{v}_{L^2(0,t;H)}^2\right)\lesssim \vert l\vert^2 \in L^1(B(0,1);\nu),
	\end{align}
	by Lemma \ref{AldousDifferences} and Remark \ref{LevyKhintchine}, we get
	\begin{align*}
		\int_{\left\{\vert l \vert\le 1\right\}}\int_0^t\vert \Real \skpH{\groupBn v_n - v_n-\left[\groupB v - v\right]}{\psi}\vert ^2\df s\nu(\df l)\to 0
	\end{align*}
	as $n\to \infty$ almost surely for all $t\in [0,T].$ For $r>1,$ we employ similar estimates as in \eqref{MajoranteInL} for
			\begin{align*}
			&\Etilde \int_0^T \left(\int_{\left\{\vert l \vert\le 1\right\}}\int_0^t\vert \Real \skpH{\groupBn v_n - v_n-\left[\groupB v - v\right]}{\psi}\vert ^2\df s\nu(\df l)\right)^r \df r\\
			&\qquad \lesssim \norm{\psi}_H^{2r}\Etilde \int_0^T \left(\norm{v_n}_{L^2(0,t;H)}^2+\norm{v}_{L^2(0,t;H)}^2\right)^r \df r\\
			&\qquad \lesssim \norm{\psi}_H^{2r} T^{1+r} \Etilde \left[\sup_{s\in [0,T]} \left(\norm{v_n}_{H}^2+\norm{v}_{H}^2\right)^r\right]<\infty,
			\end{align*}
	and thus, we get \eqref{convergenceNoiseTerm} by Vitali's Theorem. \\
	
	\emph{ad \eqref{convergenceMarcusTerm}.} From \eqref{pointwiseConvergenceStochTerm},
	\begin{align*}
		&\int_0^t\vert\Real \skpH{\im \Bn(l)v_n-\im \B(l)v}{\psi}\vert\df s\\
		&\qquad\le  \norm{\psi}_H \left(\norm{\Bn(l)(v_n-v)}_{L^1(0,t;H)}+ \norm{\left[\Bn(l)-\B(l)\right]v}_{L^1(0,t;H)}\right)\\
		&\qquad\le  \norm{\psi}_H t^\frac{1}{2}\left(\norm{\B(l)}_{\LinearOperators{H}}\norm{v_n-v}_{L^2(0,t;H)}+ \norm{\left[\Bn(l)-\B(l)\right]v}_{L^2(0,t;H)}\right)\to 0
	\end{align*}
	and the bound
	\begin{align*}
		\int_0^t\vert \Real \skpH{\groupBn v_n(s) - v_n(s) + \im \Bn(l) v_n(s)}{\psi} \vert \df s&\le \frac{1}{2} b_H \norm{\psi}_H \vert l\vert^2 \norm{v_n}_{L^2(0,t;H)}^2\\
		&\lesssim_{\omega,t} \vert l\vert^2 \in L^1(B(0,1);\nu)
	\end{align*}
	by Lemma \ref{AldousDifferences}, we infer \eqref{convergenceMarcusTerm} pointwise in $\bar{\Omega}\times [0,T].$ The $L^2(\bar{\Omega}\times [0,T])$-convergence follows similarly as in the previous step by the Vitali type argument based on the uniform bounds on $v_n,$ $n\in\N.$
\end{proof}

Finally, we are ready to summarize our results and obtain the existence of a martingale solution.

\begin{proof}[Proof of  Theorem \ref{mainTheorem}]
	\emph{Step 1.}
	Let us define the maps
	\begin{align*}
	M_{n,\psi}(w,t)=&\widetilde{u_{0,n}} -\im \int_0^t\Real \duality{ A w(s)+P_n F(w(s))}{\psi} \df s\\
	&\quad+ \int_{0}^{t} \int_{\left\{\vert l \vert\le 1\right\}} \Real \skpH{\groupBn w(s-) - w(s-)}{\psi} \tilde{\bar{\eta}}(\df s,\df l)\nonumber\\
	&\quad + \int_{0}^{t}\int_{\left\{\vert l \vert\le 1\right\}} \Real \skpH{ \groupBn w(s) - w(s) + \im \Bn(l) w(s)}{\psi}\nu(\df l)\df s;
	\end{align*}
	\begin{align*}
	M_{\psi}(w,t)=& u_0 -\im \int_0^t\Real \duality{ A w(s)+F(w(s))}{\psi} \df s\\
	&\quad+ \int_{0}^{t} \int_{\left\{\vert l \vert\le 1\right\}} \Real \skpH{\groupB w(s-) - w(s-)}{\psi} \tilde{\bar{\eta}}(\df s,\df l)\nonumber\\
	&\quad + \int_{0}^{t}\int_{\left\{\vert l \vert\le 1\right\}} \Real \skpH{ \groupB w(s) - w(s) + \im \B(l) w(s)}{\psi}\nu(\df l)\df s.
	\end{align*}
	The results of Lemma \ref{LemmaConvergences} can be summarized as
	\begin{align*}
	\Real \skpH{v_n}{\psi}-M_{n,\psi}(v_n,\cdot)\to \Real \skpH{v}{\psi}-M_{\psi}(v,\cdot),\qquad n\to \infty,
	\end{align*}
	in $L^2(\bar{\Omega}\times [0,T])$ for all $\psi\in \EA$ and from the definition of $u_n$ via the Galerkin equation, we infer
	$\Real \skpH{u_n(t)}{\psi}=M_{n,\psi}(u_n,t)$ almost surely for all $t\in [0,T].$
	Due to the identity $\Leb_{[0,T]}\otimes \Prob^{u_n}=\Leb_{[0,T]}\otimes \bar{\Prob}^{v_n}$, we obtain
	\begin{align*}
	\Etilde\int_0^T \vert\Real \skpH{v(t)}{\psi}-M_{\psi}(v,t)\vert^2 \df t=& \lim_{n\to \infty} \Etilde\int_0^T \vert\Real \skpH{v_n(t)}{\psi}-M_{n,\psi}(v_n,t)\vert^2 \df t\\
	=& \lim_{n\to \infty} \E\int_0^T \vert\Real \skpH{u_n(t)}{\psi}-M_{n,\psi}(u_n,t)\vert^2 \df t=0
	\end{align*}
	and thus,
	\begin{align*}
		\bar{\Prob}\left\{\Real \skpH{v(t)}{\psi}=M_{\psi}(v,t)\quad \text{f.a.a. $t\in[0,T]$}\right\}=1.
	\end{align*}
	Since both $\Real \skpH{v}{\psi}$ and $M_{\psi}(v,\cdot)$ are almost surely in $\D([0,T]),$ we obtain
		\begin{align*}
		\bar{\Prob}\left\{\Real \skpH{v(t)}{\psi}=M_{\psi}(v,t)\quad \forall t\in[0,T]\right\}=1,
		\end{align*}
	which means that $\left(\bar{\Omega},\bar{\F},\bar{\Prob},\bar{\eta},\bar{\Filtration},u\right)$ is a martingale solution to $\eqref{marcus2}.$\\

	\emph{Step 2.} In order to conclude the proof, we need to show that the process $v$ satisfies the mass preservation condition \eqref{eqn-normPreservation}.
Let us first fix $\omega \in \bar{\Omega}$ such that 
\begin{align}\label{vkConvergenceZT}
	v_k (\cdot,\omega) \to v(\cdot,\omega) \mbox{ in } Z_T,
\end{align}
as $k\to\infty.$
By part (a)(ii) of Corollary \ref{cor-main}, the set of such elements is a full set in $\bar\Omega$. Together with Lemma \ref{lem-boundednessWeaklyCadlag}, \eqref{vkConvergenceZT} implies that there exists $r=r(\omega)>0$ such that $\sup_{t\in[0,T]}\norm{v_k(t,\omega)}\le r$ for every $k\in\N.$  From \eqref{vkConvergenceZT} and Proposition \ref{cadlagTopology}, we infer 
that there is a sequence $\left(\lambda_k\right)_{n\in\N}=\left(\lambda_k(\omega)\right)_{n\in\N}\in \Lambda^\N,$ such that 
\begin{align*}
\sup_{t\in[0,T]}\norm{v_k(\lambda_k(t),\omega)-v(t,\omega)}_{\EAdual}\to 0,\quad k\to\infty.
\end{align*}
Hence, we get
\begin{align*}
&\sup_{t\in[0,T]}\norm{v_k(\lambda_k(t),\omega)-v(t,\omega)}_{H}\\&\qquad\lesssim \sup_{t\in[0,T]}\left[\norm{v_k(\lambda_k(t),\omega)-v(t,\omega)}_{\EAdual}^\frac{1}{2}\norm{v_k(\lambda_k(t),\omega)-v(t,\omega)}_{\EA}^\frac{1}{2}\right]\\
&\qquad\le (2r)^\frac{1}{2} \sup_{t\in[0,T]}\norm{v_k(\lambda_k(t),\omega)-v(t,\omega)}_{\EAdual}^\frac{1}{2}\to 0,\qquad k\to\infty.
\end{align*}
In view of Proposition \ref{cadlagTopology}, this implies $v_k(\cdot,\omega)\to v(\cdot,\omega)$ in $\D([0,T],H)$ as $k\to\infty.$
%
%
%
%
%
%
%
%
%
%
%
Since the norm function $\norm{\cdot}_H: H \to \mathbb{R}$ is Lipschitz continuous we deduce that 
\[
\norm{v_k (\cdot,\omega)}_H \to \norm{v(\cdot,\omega)}_H \mbox{ in }  \D([0,T],\mathbb{R}).
\]
On the other hand, by  Corollary \ref{cor-main2},  we infer that 
\[
\norm{v_k (t,\omega)}_H=\norm{u_0}_H \mbox{ for all } t \in [0,T].
\]
Applying finally Lemma \ref{lem-App C-01} we infer that 
\[
\norm{v (t,\omega)}_H=\norm{u_0}_H \mbox{ for all } t \in [0,T].
\]

\end{proof}

\newpage
\begin{appendix}
	\section{Time Homogeneous Poisson Random Measure}\label{PRM}
	Let $\bar{\mathbb{N}}$ denote the set of extended natural numbers, i.e., $\bar{\mathbb{N}}:=\mathbb{N} \cup \{\infty\}$ and $\mathbb{R}^{+}:=[0,\infty).$ Let $(S, \mathscr{S})$ be a measurable space and $M_{\bar{\mathbb{N}}}(S)$ be the set of all $\bar{\mathbb{N}}$-valued measures on the measurable space $(S, \mathscr{S}).$
	On the set $M_{\bar{\mathbb{N}}}(S)$ we consider the $\sigma$-field $\mathscr{M}_{\bar{\mathbb{N}}}(S)$ defined as the smallest $\sigma$-field such that for all $C\in \mathscr{S}:$ the map $$i_C:M_{\bar{\mathbb{N}}}(S) \ni \mu \rightarrow \mu(C) \in \bar{\mathbb{N}}$$ is measurable.
	
	\begin{definition}
		Let $(Y, \mathscr{B}(Y))$ be a measurable space. A  time homogeneous Poisson random measure $\eta$ on $(Y, \mathscr{B}(Y))$ over $(\Omega,\mathcal{F},\mathbb{F},\mathbb{P})$ is a measurable function
		$$ \eta\,:\, (\Omega,\mathcal{F}) \rightarrow (M_{\bar{\mathbb{N}}}(\mathbb{R}^+ \times Y),\mathscr{M}_{\bar{\mathbb{N}}}(\mathbb{R}^+ \times Y))$$
		such that
		\begin{itemize}
			\item[(a)] for each $C\in\mathscr{B}(\mathbb{R}^+) \otimes \mathscr{B}(Y), \eta(C):=i_C \circ \eta :\Omega \rightarrow \bar{\mathbb{N}}$ is a Poisson random variable with parameter $\mathbb{E}[\eta(C)];$
			\item[(b)] $\eta$ is independently scattered, i.e., if the sets $C_1,C_2,\ldots,C_n \in \mathscr{B}(\mathbb{R}^+)\otimes \mathscr{B}(Y) $ are disjoint, then the random variables $\eta(C_1), \eta(C_2),\ldots, \eta(C_n)$ are mutually independent;
			\item[(c)] for all $U \in \mathscr{B}(Y)$ the $\bar{\mathbb{N}}-$valued process $(N(t,U))_{t \geq 0}$ defined by
			$$ N(t,U):= \eta((0,t]\times U),\quad t \geq 0$$  is $\mathcal{F}_t$-adapted and its increments are independent of the past, i.e., if $t>s \geq 0,$ then $N(t,U) - N(s,U) = \eta((s,t]\times U)$ is independent of $\mathcal{F}_s.$
		\end{itemize}
	\end{definition}
	
	If $\eta$ is a time homogeneous Poisson random measure then the formula
	$$\nu(A):=\mathbb{E}[\eta((0,1] \times A)],\quad A \in \mathscr{B}(Y)$$ defines a measure on $(Y, \mathscr{B}(Y))$ called the intensity measure of $\eta.$ We assume that $\nu$ is $\sigma$-finite.
	Moreover, for all $T<\infty$ and all $A \in \mathscr{B}(Y)$ such that $\mathbb{E}[\eta((0,T] \times A)]< \infty,$ the $\mathbb{R}-$valued process $\{\tilde{N}(t,A)\}_{t \in [0,T]}$ defined by
	$$\tilde{N}(t,A):=\eta((0,t]\times A) -t \, \nu(A),\quad t \in (0,T],$$
	is an integrable martingale on $(\Omega,\mathcal{F},\mathbb{F},\mathbb{P}).$ The random measure $m\otimes\nu$ on $\mathscr{B}(\mathbb{R}^+)\otimes \mathscr{B}(Y),$ where $m$ stands for the Lebesgue measure (often denoted also as $\Leb$), is called a compensator of $\eta$ and the difference between a time homogeneous Poisson random measure $\eta$ and its compensator, i.e.,
	$$ \tilde{\eta}:=\eta-m\otimes\nu,$$ is called a compensated time homogeneous Poisson random measure.
	
	We follow the notion of Ikeda and Watanabe \cite{IkedaWatanabe}, Peszat and Zabczyk \cite{PeszatZabczyk},  to list some of the basic properties of the stochastic integral with respect to $\tilde{\eta}.$
	Let $E$ be a separable Hilbert space and let $\mathcal{P}$ be a predictable $\sigma$-field on $[0,T]\times\Omega.$ Let $\mathfrak{L}^2_{\nu,T}(\mathcal{P}\otimes\mathscr{B}(Y), m \otimes \mathbb{P}\otimes \nu; E)$ be a space of all $E$-valued, $\mathcal{P}\otimes\mathscr{B}(Y)$-measurable processes such that $$\mathbb{E}\bigg[\int_0^T \int_Y \|\xi(s,\cdot,y)\|_E^2 \, d \nu(y) \,ds \bigg] <\infty. $$
	
	If $\xi \in \mathfrak{L}^2_{\nu,T}(\mathcal{P}\otimes\mathscr{B}(Y), m \otimes \mathbb{P}\otimes \nu; E)$ then the integral process ${\int_0^T \int_Y} \xi(s,\cdot,y)\tilde{\eta}(ds,dy),$ $t \in[0,T],$ is a c\`{a}dl\`{a}g square-integrable $E$-valued martingale. Moreover, we have the following isometry formula
	\begin{align}\label{isof}
	\mathbb{E}\bigg[\bigg\| \int_0^T \int_Y \xi(s,\cdot,y) \,\tilde{\eta}(ds,dy)\bigg\|_E^2\bigg]=\mathbb{E}\bigg[\int_0^T \int_Y \|\xi(s,\cdot,y)\|_E^2 \, d \nu(y) \,ds \bigg],\, t\in[0,T].
	\end{align}
	
	\section{Marcus Canonical SDEs}
	Let $(\Omega, \mathcal{F}, \mathbb{F}, \mathbb{P})$ be a probability space equipped with a filtration $\mathbb{F}:=\left\lbrace \mathcal{F}_t, t \geq 0\right\rbrace$ that satisfies the usual hypothesis (i.e., $\mathcal{F}_0$ contains all $\mathbb{P}$-null sets and $\mathcal{F}$ is right continuous). Let $\v_0, \v_1, \ldots, \v_N: \mathbb{R}^d\to\mathbb{R}^d$ be complete $C^1$-vector fields. Define $\v: \mathbb{R}^d\to\mathcal{L}(\mathbb{R}^N, \mathbb{R}^d)$ such that $\v(y)(h):=\sum_{j=1}^N \v_j(y)h_j, h\in\mathbb{R}^N, y\in\mathbb{R}^d$.

	Let $L(t) := (L_1(t), \cdots, L_N(t))$ be a $\mathbb{R}^N-$ valued L\'evy process with pure jump,
	\begin{equation*}
	L(t) = \int_{0}^{t}\int_{B}\!l \,\tilde{\eta}(ds,dl) + \int_{0}^{t}\int_{B^c}\!l \,\eta(ds,dl)
	\end{equation*} where $B := \mathbb{B}(0,1) \subset \mathbb{R}^N$, $l=(l_1,\ldots, l_N)\in\mathbb{R}^N$ ; $\eta$, $\tilde{\eta}$ represent homogeneous Poisson random measure and the compensated one with the compensator $m\otimes\nu$ respectively. We always assume that $\eta$ is independent of $\mathcal{F}_0$.
	
	Consider the following ``Marcus" stochastic differential equation:
	\begin{align}\label{MSDE}
	dY(t)&=\v_0(Y(t))\,dt +\v(Y(t-))\diamond dL(t)\nonumber\\
	&=\v_0(Y(t))\,dt + \sum_{j=1}^N\v_j(Y(t-))\diamond dL_j(t),
	\end{align}
	which is defined in the integral form as follows
	\begin{align}\label{MSDE1}
	Y(t)&=Y_0+\int_0^t\v_0(Y(s))\,ds +\int_0^t\int_B\Big[\Phi\big(1,l,Y(s-)\big)-Y(s-)\Big]\tilde{\eta}(ds, dl)\nonumber\\
	&\quad\quad+\int_0^t\int_{B^c}\Big[\Phi\big(1,l,Y(s-)\big)-Y(s-)\Big] \eta(ds, dl)\nonumber \\
	&\quad\quad+\int_0^t\int_B\Big[\Phi\big(1,l,Y(s)\big)-Y(s)-\sum_{j=1}^N l_j\v_j(Y(s))\Big] \nu(dl) ds,	
	\end{align}
	where $y(t):=\Phi(t,l,y_0)$ solves
	\begin{align}\label{ODE}
	\frac{dy}{dt} = \sum_{j=1}^Nl_j\v_j(y), \quad \ \textrm{ with initial condition}\ y(0)=y_0.
	\end{align}
	\noindent

	\begin{theorem}[It\^o's formula 1]\label{T1}
		Let $\varphi:\mathbb{R}^d\to\mathbb{R}^k$ is a $C^1$-class function.
		If $Y$ is an $\mathbb{R}^d$-valued process a solution to \eqref{MSDE}, then
		\begin{align}\label{eqn-Ito2}	
		&\varphi(Y(t))-\varphi(Y_0)\nonumber\\
		&= \int_0^t \varphi^\prime (Y(s))(\v_0(Y(s)))\, ds+\int_0^t\int_{B^c}\Big[\varphi\big(\Phi\big(1,l,Y(s-)\big)\big)-\varphi(Y(s-))\Big]	\eta(ds, dl)\nonumber\\
		&\quad\quad+\int_0^t\int_{B}\Big[\varphi\big(\Phi\big(1,l,Y(s-)\big)\big)-\varphi(Y(s-))\Big]	\tilde{\eta}(ds, dl)
		\nonumber\\
		&\quad\quad+\int_0^t\int_{B}\Big[\varphi\big(\Phi\big(1,l,Y(s)\big)\big)-\varphi(Y(s))-\sum_{j=1}^N l_j\varphi'(Y(s))(\v_j(Y(s)))\Big]\nu(dl) ds.
		\end{align}

		Moreover, when $k=d$ and $\varphi:\mathbb{R}^d\to\mathbb{R}^d$ is a $C^1$-diffeomorphism, we define for each $j=0,1,\ldots,N$,
		the ``Push-forward" of the vector fields $\v_j$ by $\varphi'$ as $\hat{\v}_j: \mathbb{R}^d\to\mathbb{R}^d$ such that
		\[z\mapsto (d_{_{\varphi^{-1}(z)}}\varphi)\big(\v_j(\varphi^{-1}(z))\big) :=\varphi'(\varphi^{-1}(z)) \big(\v_j(\varphi^{-1}(z))\big).\]
		Let $\hat{\v}: \mathbb{R}^d\to\mathcal{L}(\mathbb{R}^N, \mathbb{R}^d)$ be as before. \\
		Then $Y$ is a solution to \eqref{MSDE} iff $$Z(t):=\varphi(Y(t))$$ is a solution to
		\begin{equation}\label{MSDE2}
		dZ=\hat{\v}_0(Z(t))\,dt + \hat{\v}(Z(t))\diamond dL(t), \quad Z_0=\varphi(Y_0).
		\end{equation}
	\end{theorem}
	
	We will now present an infinite dimensional version of the above result, which has been used in this work.
	\par
	\noindent
	As before let $(\Omega, \mathcal{F}, \mathbb{F}, \mathbb{P})$ be a complete probability space. Let $E$ be a separable Hilbert space. Let $\v_0, \v_1, \ldots, \v_N: E\to E$ be complete $C^1$-vector fields. Define $\v: E\to\mathcal{L}(\R^N, E)$ such that $\v(y)(h):=\sum_{j=1}^N \v_j(y)h_j, h\in\R^N, y\in E$. Define the L\'evy process $L(t)$ as before. Define the Marcus mapping
	$$\Phi: \mathbb{R}_{+}  \times \R^N \times E \rightarrow E$$ such that for each fixed $l \in \R^N$, $y_0 \in E$, the function $$t \mapsto \Phi(t,l,u_0)$$ is the continuously differentiable solution of the ordinary differential equation
	\begin{equation*}
	\dfrac{dy}{dt} = \sum_{j=1}^N l_j v_j(y),\ t\geq 0,
	\end{equation*} with $y(0)=y_0 \in E$, and $l=(l_1, l_2, \ldots, l_N)\in B$, i.e.,
	\begin{equation*}
	\Phi(t,l,y_0) = \Phi(0,l,y_0) + \int_{0}^{t}\! \sum_{j=1}^N l_j v_j(\Phi(s,l,y_0) )\, ds, \ t\geq 0.
	\end{equation*}
	With the above setting, let us consider the $E$-valued process $Y$ given by \eqref{MSDE1}. Then we have the following result.
	
	\begin{theorem}[It\^o's formula 2]\label{T2}
		Let $G$ be a separable Hilbert space and $\varphi:E\to G$ be a $C^1$-class function such that the first derivative $\varphi^\prime: E\to \mathcal{L}(E, G)$ is $(p-1)$-H\"older continuous.
		If $Y$ is an $E$-valued process given by \eqref{MSDE1}, then for every $t>0$, we have $\Prob$-a.s.
		\begin{align}\label{eqn-Ito2}	
		&\varphi(Y(t))-\varphi(Y_0)\nonumber\\
		&= \int_0^t \varphi^\prime (Y(s))(\v_0(Y(s)))\, ds+\int_0^t\int_{B^c}\Big[\varphi\big(\Phi\big(1,l,Y(s-)\big)\big)-\varphi(Y(s-))\Big]	\eta(ds, dl)\nonumber\\
		&\quad\quad+\int_0^t\int_{B}\Big[\varphi\big(\Phi\big(1,l,Y(s-)\big)\big)-\varphi(Y(s-))\Big]	\tilde{\eta}(ds, dl)
		\nonumber\\
		&\quad\quad+\int_0^t\int_{B}\Big[\varphi\big(\Phi\big(1,l,Y(s)\big)\big)-\varphi(Y(s))-\sum_{j=1}^N l_j\varphi'(Y(s))(\v_j(Y(s)))\Big]\nu(dl) ds.
		\end{align}

		Moreover, when $\varphi:E\to E$ is a $C^1$-diffeomorphism, we define for each $j=0,1,\ldots,N$,
		the ``Push-forward" of the vector fields $\v_j$ by $\varphi'$ as $\hat{\v}_j: E\to E$ such that
		\[z\mapsto (d_{_{\varphi^{-1}(z)}}\varphi)\big(\v_j(\varphi^{-1}(z))\big) :=\varphi'(\varphi^{-1}(z)) \big(\v_j(\varphi^{-1}(z))\big).\]
		Let $\hat{\v}: E\to\mathcal{L}(\R^N, E)$ be as before. \\
		Then $Y$ is a solution to \eqref{MSDE} iff $$Z(t):=\varphi(Y(t))$$ is a solution to
		\begin{equation}\label{MSDE2}
		dZ=\hat{\v}_0(Z(t))\,dt + \hat{\v}(Z(t))\diamond dL(t), \quad Z_0=\varphi(Y_0).
		\end{equation}
	\end{theorem}
	\begin{proof}
		Let us assume, for the sake of simplicity, $\eta=0$ on $B^c$. For $y\in E$, define
		\begin{align}
		f(y,l)&:=\Phi\big(1,l,y\big)-y\ \ \mbox{for all}\ l\in B\label{defnf}\\
		a(y)&:=\v_0(y) + \int_B\Big[\Phi\big(1,l,y\big)-y-\sum_{j=1}^N l_j\v_j(y)\Big] \nu(dl)\nonumber \\
		&= \v_0(y) + \int_B\Big[f(y,l)-\sum_{j=1}^N l_j\v_j(y)\Big] \nu(dl).\label{defna}
		\end{align}
		Then the $E$-valued process $Y$ given in \eqref{MSDE1} takes the form
		\begin{align}\label{MSDE2}
		Y(t)&=Y_0+\int_0^t a(Y(s))\,ds +\int_0^t\int_B f(Y(s-),l) \tilde{\eta}(ds, dl).
		\end{align}
		Then by the It\^o's formula (see Theorem B.1 in Brze\'zniak {\it et al.} \cite{BrzezniakHausenblasJumpNavierStokes}), for every $t>0$, we have $\Prob$-a.s.
		\begin{align}\label{eqn-Ito3}	
		\varphi(Y(t))&=\varphi(Y_0)+\int_0^t \varphi^\prime (Y(s))(a(Y(s)))\, ds+\int_0^t\int_B \varphi^\prime (Y(s-)) (f(Y(s-),l)) \tilde{\eta}(ds, dl)\nonumber\\
		&\quad + \int_0^t\int_B\Big[\varphi\big(Y(s-)+f(Y(s-),l)\big) - \varphi(Y(s-)) - \varphi^\prime (Y(s-)) (f(Y(s-),l)) \Big] \eta(ds, dl)\nonumber\\
		&= \varphi(Y_0)+\sum_{i=1}^3 I_i.
		\end{align}
		Note that by the definition of $a$ in \eqref{defna}
		\begin{align}\label{I1}
		I_1&:= \int_0^t \varphi^\prime (Y(s))(a(Y(s)))\, ds\nonumber\\
		&=\int_0^t \varphi^\prime (Y(s))(\v_0(Y(s)))\, ds + \int_0^t\int_B \varphi^\prime (Y(s))(f(Y(s),l))\, \nu(dl)\, ds\nonumber\\
		&\quad- \int_0^t\int_B \Big[\sum_{j=1}^N l_j\varphi'(Y(s))(\v_j(Y(s)))\Big] \nu(dl)\, ds.
		\end{align}
		Using the definitions of $f$ in \eqref{defnf} and that of compensated Poisson random measure $ \tilde{\eta}:=\eta-m\otimes\nu,$ we have
		\begin{align}\label{I3}
		I_3&:= \int_0^t\int_B\Big[\varphi\big(Y(s-)+f(Y(s-),l)\big) - \varphi(Y(s-)) - \varphi^\prime (Y(s-)) (f(Y(s-),l)) \Big] \eta(ds, dl)\nonumber\\
		&=\int_0^t\int_B\Big[\varphi\big(\Phi(1,l,Y(s-))\big) - \varphi(Y(s-)) - \varphi^\prime (Y(s-)) (f(Y(s-),l)) \Big] \eta(ds, dl)\nonumber\\
		&= \int_0^t\int_B\Big[\varphi\big(\Phi(1,l,Y(s-))\big) - \varphi(Y(s-))\Big] \tilde{\eta}(ds, dl) - \int_0^t\int_B \varphi^\prime (Y(s-)) (f(Y(s-),l)) \tilde{\eta}(ds, dl)\nonumber\\
		&\quad +  \int_0^t\int_B\Big[\varphi\big(\Phi(1,l,Y(s))\big) - \varphi(Y(s))\Big] \nu(dl)\, ds - \int_0^t\int_B \varphi^\prime (Y(s))(f(Y(s),l))\, \nu(dl)\, ds.
		\end{align}
		Note that, while adding up $I_1, I_2$ and $I_3$, the second term of \eqref{I1} and the last term of \eqref{I3} cancel each other. Also note the 2nd term on the right hand of  \eqref{I3} is $-I_2$, and thus it gets cancelled with $I_2$. Hence using \eqref{I1} and \eqref{I3} in \eqref{eqn-Ito3}, and grouping the similar integrals we have the desired result \eqref{eqn-Ito2}.	
		\par
		\noindent
		To prove the second part of the Theorem, let us define a map
		$$\hat{\Phi}:\mathbb{R}_+\times\R^N\times E\to E$$ such that for all $l\in\R^N, z\in E$, the function $t\mapsto\hat{\Phi}(t,l,z)$ solves
		\begin{align*}
		\frac{dz}{dt} = \sum_{j=1}^Nl_j\hat{\v}_j(z),\ \ t\geq 0, \ \ z(0)=z.
		\end{align*}
		Let us assume that $\varphi:E\to E$ is a $C^1$-diffeomorphism. Then one can show that for all $l\in\R^N$ and $t\geq 0$
		\begin{align}\label{P1}
		\hat{\Phi}(t,l,z) = \varphi\big(\Phi(t, l, y)\big)\ \textrm{where}\ z=\varphi(y),\ y\in E.
		\end{align}	
		Then from the It\^o's formula \eqref{eqn-Ito2}, we deduce
		\begin{align*}	
		Z(t)&=Z_0+\int_0^t \hat{\v}_0(Z(s))\, ds+\int_0^t\int_{B^c}\Big[\hat{\Phi}\big(1,l,Z(s-)\big)-Z(s-)\Big] \eta(ds, dl)\\
		&\quad\quad+\int_0^t\int_{B}\Big[\hat{\Phi}\big(1,l,Z(s-)\big)-Z(s-)\Big]\tilde{\eta}(ds, dl)\\
		&\quad\quad+\int_0^t\int_{B}\Big[\hat{\Phi}\big(1,l,Z(s)\big)-Z(s)-\sum_{j=1}^N l_j\hat{\v}_j(Z(s))\Big]\nu(dl) ds.
		\end{align*}
		This proves $Z(t)=\varphi(Y(t))$ is an $E$-valued process satisfying
		$$dZ=\hat{\v}_0(Z(t))\,dt + \hat{\v}(Z(t))\diamond dL(t),\ \  Z_0=\varphi(Y_0).$$
		Converse part can similarly be proven.
	\end{proof}

	\section{A simple convergence result}\label{sec-App C}
\begin{lemma}\label{lem-App C-01} Suppose that for each $n\in\mathbb{N}$, a function  $f_n \in \mathbb{D}([0,T],\mathbb{R})$ is constant and that for some
	$f \in \mathbb{D}([0,T],\mathbb{R})$, 
	$f_n \to f$ in $\mathbb{D}([0,T],\mathbb{R})$ as $n\to \infty.$ Then $f$ is also a constant function and $f_n \to f$ in $C([0,T],\mathbb{R})$ as $n\to \infty.$ 	
\end{lemma}
\begin{proof}
	Let us denote, for each $n\in\mathbb{N}$, the value of the function $f_n$ by $c_n$, for some $c_n \in \mathbb{R}$. By 
	part (b) of Proposition \ref{cadlagTopology}
	there exists a sequence  $\left(\lambda_n\right)\in \Lambda^\N$  such that
	\begin{align}\label{eqn-App C-01}
	\sup_{t\in [0,T]} \vert \lambda_n(t)-t\vert \to 0
	\end{align}
	and 
	\begin{align}\label{eqn-App C-02}
	\sup_{t\in [0,T]} \vert f_n(\lambda_n(t))-f(t) \vert \to 0,\qquad n\to \infty.
	\end{align}
	This yields 
	\begin{align*}
	\sup_{t\in [0,T]} \vert f_n(t)-f(t) \vert =\sup_{t\in [0,T]} \vert c_n-f(t) \vert=\sup_{t\in [0,T]} \vert f_n(\lambda_n(t))-f(t) \vert \to 0,\qquad n\to \infty.
	\end{align*}
	Moreover, \eqref{eqn-App C-02} implies 
		\begin{align*}
		\vert f(t)-f(s)\vert\le& \vert f(t)-c_n\vert+\vert c_n-f(s)\vert=\vert f(t)-f_n(\lambda_n(t))\vert+\vert f_n(\lambda_n(s))-f(s)\vert\to 0
%
		\end{align*}
		as $n\to \infty$ 	
		for $s,t\in[0,T].$ Hence, $f$ is a constant function as claimed. 
%
%
%
%
%
%
%
%
\end{proof}
We conclude this section with the following result. 
\begin{corollary}\label{cor-App C-02} 
	Let $n\in\N_0$ and $c\geq 0.$ Then, the set 
	\[
	\mathcal{S}=\bigl\{u \in \mathbb{D}([0,T],H_n): \norm{u(t)}_H=c \mbox{ for all } t\in [0,T] \bigr\}
	\]
	is closed in $\mathbb{D}([0,T],H_n)$.
\end{corollary}
\begin{proof} Take an $\mathcal{S}$-valued sequence $(u_k)$ such that $u_k \to u$ in $\mathbb{D}([0,T],H_n)$ for some $u \in \mathbb{D}([0,T],H_n)$. For $t\in [0,T]$ and $k\in\N,$ we define $f_k(t)=\norm{u_k(t)}_H$ and 
	$f(t)=\norm{u(t)}_H.$
	Since the $H$-norm function is Lipschitz on $H_n$, we infer that $f_k \to f$ in $\mathbb{D}([0,T],\mathbb{R})$.  In view of Lemma \ref{lem-App C-01} we obtain $f(t)=c$ for all $t\in[0,T]$ which implies $u \in \mathcal{S}$.
	
\end{proof}

\end{appendix} 
%

\par\medskip\noindent

\end{document}